\documentclass[11pt,oneside]{amsart}

      \usepackage{amssymb}
      \usepackage{graphicx}
      \usepackage{dcolumn,indentfirst,cite,color,hyperref}

      \theoremstyle{plain}
     \newtheorem{thm}{Theorem}[section]
\newtheorem{lem}[thm]{Lemma}

\newtheorem{cor}[thm]{Corollary}
\newtheorem{pro}[thm]{Proposition}
\newtheorem{rmk}[thm]{Remark}
\newtheorem{ex}[thm]{Example}
\newtheorem{defi}[thm]{Definition}

\newcommand{\bess}{\begin{eqnarray*}}
\newcommand{\eess}{\end{eqnarray*}}
\newcommand{\Int}{{\rm Int}}
\input xy
\xyoption{all}

\usepackage{fancyhdr} %ҳüºÍÒ³½Å
\begin{document}

% First we specify the top matter (author, title, etc).
%
% Note: All of the top matter items are optional and can be omitted.
% But you will probably want to specify at least the author and title% and perhaps an abstract.

   % author information

   % first author

   % second author

\author{Yuefei Wang}
\address{Yuefei Wang, Institute of Mathematics, AMSS, Chinese Academy of Sciences, Beijing, 100190, P.R.China}
\email{wangyf@math.ac.cn}

\author{Jinghua Yang}
\address{Jinghua Yang, Shanghai Univ., Shanghai, 200444, P.R.China}
\email{davidyoung@amss.ac.cn}

%%%%%%%%%%%%%%%%%%%%%%%%%%%%%%%%%%%%%%%%%%%%%%%%%%%%%

  \title[]{\bf\ On $p$-adic M\"obius maps}

   % Note that the short title for running heads goes in square
   % brackets.  This is optional.  The long title goes in curly
   % braces.  In the long title, line breaks are indicated by \\.

   % abstract (optional)
   \begin{abstract} In this paper, we study three aspects of the $p-$adic M\"obius maps. One is the group $\mathrm{PSL}(2,\mathcal{O}_{p})$, another is the geometrical characterization of the $p-$adic M\"obius maps and its application, and the other is different norms of the $p-$adic M\"obius maps. Firstly, we give a series of equations of the $p-$adic M\"obius maps in $\mathrm{PSL}(2,\mathcal{O}_{p})$ between matrix, chordal, hyperbolic and unitary aspects. Furthermore, the properties of $\mathrm{PSL}(2,\mathcal{O}_{p})$ can be applied to study the geometrical characterization, the norms, the decomposition theorem of $p-$adic M\"obius maps, and the convergence and divergence of $p-$adic continued fractions. Secondly, we classify the $p-$adic M\"obius maps into four types and study the geometrical characterization of the $p-$adic M\"obius maps from the aspects of fixed points in $\mathbb{P}^{1}_{Ber}$ and the invariant axes which yields the decomposition theorem of $p-$adic M\"obius maps. Furthermore, we prove that if a subgroup of $\mathrm{PSL}(2,\mathbb{C}_{p})$ containing elliptic elements only, then all elements fix the same point in $\mathbb{H}_{Ber}$ without using the famous theorem--Cartan fixed point theorem, and this means that this subgroup has potentially good reduction. In the last part, we extend the inequalities obtained by Gehring and Martin\cite{F.G1,F.G2},  Beardon and Short \cite{AI} to the non-archimedean settings.  These inequalities of $p$-adic M\"obius maps are between the matrix, chordal, three-point and unitary norms. This part of work can be applied to study the convergence of the sequence of $p-$adic M\"obius maps which can be viewed as a special cases of the work in \cite{CJE} and the discrete criteria of the subgroups of $\mathrm{PSL}(2,\mathbb{C}_{p})$.

   %all the elements in  are chordal isometries, fix the Guass point in the hyperbolic Berkovich space, have the matrix norm $1$ and do not change

% Moreover, we also give a complete characterization of this boundary.
   \end{abstract}

   % AMS subject classifications (used in AMS journals)

   % AMS keywords (used in AMS journals)

   % acknowledge support, etc
  % \thanks{This research was partially supported by NSF grant
   %  DOA-123456789.}
   %\thanks{We would like to thank our colleagues for their helpful
    % criticism.}

   % dedication
  % \dedicatory{Dedicated to Professor Donald Knuth on the occasion
   %  of his $100$th birthday}

   % today's date, or fill in whatever date you prefer
   \date{\today}

% This ends the top matter information.
% We can now tell LaTeX to display the top matter.

   \maketitle
%\pagestyle{myheadings}%Ò³ÂëλÓÚÒ³Ã涥¶Ë

%\maketitle
%%%%%%%%%%%%%%%%%%%%%%%%%%%%%%%%%%%%%%%%%%%%%%%%%%%%%%%%%%%%%%%%%%%%%%%%%%%%%%%%%%%%%%%%%%%%%%%%%%%%%%%%%%%%%%%%%%%%%

%\vskip1.0cm
%\tableofcontents
%\footnote[]{August 30, 2011}
%%%%%%%%%%%%%%%%%%%%%%%%%%%%%%%%%%%%%%%%%%%%%%%%%%%%%%%%%%%%%%%%%%%%%%%%%%%%%%%%%%%%%%%%%%%%%%%%%%%%%%%%%%%%%%%%%%%%%
%%%%%%%%%%%%%%%%%%%%%%%%%%%%%%%%%%%%%%%%%%%%%%%%%%%%%%%%%%%%%%%%%%%%%%%%%%%%%%%%%%%%%%%%%%%%%%%%%%%%%%%%%%%%%%%%%%%%%

%%%%%%%%%%%%%%%%%%%%%%%%%%%%%%%%%%%%%%%%%%%%%%%%%%%%%%%%%%%%%%%%%%%%%%%%%%%%%%%%%%%%%%%%%%%%%%%%%%%%%%%%%%%%%%%%%%%%

%\vskip1.0cm
%\tableofcontents
%%%%%%%%%%%%%%%%%%%%%%%%%%%%%%%%%%%%%%%%%%%%%%%%%%%%%%%%%%%%%%%%%%%%%%%%%%%%%%%%%%%%%%%%%%%%%%%%%%%%%%%%%%%%%%%%%%%%%
%%%%%%%%%%%%%%%%%%%%%%%%%%%%%%%%%%%%%%%%%%%%%%%%%%%%%%%%%%%%%%%%%%%%%%%%%%%%%%%%%%%%%%%%%%%%%%%%%%%%%%%%%%%%%%%%%%%%%

\section{Introduction}

\subsection{Statement of results }

We call an element in the projective special linear group $\mathrm{PSL}(2,\mathbb{C}_{p})$ the $p$-adic M\"obius map, where $\mathbb{Q}_{p}$ is the field of $p$-adic rational
numbers and $\mathbb{C}_{p}$ is the completion of the algebraic closure of
$\mathbb{Q}_{p}$. The projective space $\mathbb{P}^{1}(\mathbb{C}_{p})$ is totally disconnected and not locally compact, which implies that we can not adopt the method used in dealing with the complex settings easily. The main tool that we use is the projective Berkovich space $\mathbb{P}^{1}_{Ber}$(see concrete definitions in section 3), since $\mathrm{PSL}(2,\mathbb{C}_{p})$ acts on the hyperbolic Berkovich space $\mathbb{H}_{Ber}$ isometrically and the projective Berkovich space is compact with respect to the weak topology.

We study the subgroup $\mathrm{PSL}(2,\mathcal{O}_{p})$, where $\mathcal{O}_{p}=\{z||z|\le 1\}$ firstly, since the study of the unitary groups of the projective special linear group $\mathrm{PSL}(2,\mathbb{C})$ is a very important topic in the study of M\"obius maps. There exist a lot of equations of M\"obius maps in the unitary group. It is natural to generalize the equations to the non-archimedean settings.   We give a series of equations of the $p-$adic M\"obius maps in $\mathrm{PSL}(2,\mathcal{O}_{p})$ between matrix, chordal, hyperbolic and unitary aspects. Furthermore, the properties of $\mathrm{PSL}(2,\mathcal{O}_{p})$ can be applied to study the geometrical characterization, the norms and the decomposition theorem of $p-$adic M\"obius maps and the convergence and divergence of $p-$adic continued fractions.

Let $g$ be a $p$-adic M\"obius map, $\rho_{v}(z,w)$ be the chordal metric on the projective line $\mathbb{P}^{1}(\mathbb{C}_{p})$, $\rho(z,w)$ be the hyperbolic metric on the hyperbolic Berkovich space, $L(g)=p^{\rho(g(\zeta_{Gauss}),\zeta_{Gauss})}$, and $\parallel \cdot \parallel$ be the matrix norm (see concrete definitions in section 2).

\begin{thm}\label{thm:1}
For any $p$-adic M\"obius map $g$, the following are equivalent:

$(1)$ $g\in \mathrm{PSL}(2,\mathcal{O}_{p})$;

$(2)$ $L(g)=1$;

$(3)$  $\rho(\zeta_{Gauss},g(\zeta_{Gauss}))=0$;

$(4)$ $g$ is a chordal isometry;

$(5)$ $\parallel g\parallel=1$;

$(6)$ for any $h$ in the Matrix ring $\mathrm{M}(2,\mathbb{C}_{p})$, $\parallel gh\parallel=\parallel hg\parallel=\parallel ghg^{-1}\parallel=\parallel h\parallel$.

\end{thm}

In the second part, we study the characterization of $p-$adic M\"obius maps. In \cite{K.F}, Kato  introduced the idea of the Klenian group to study the $p-$adic M\"obius maps firstly, and in \cite{J.V,QY}, Vermitage and Parker, Qiu, Yang and Yin gave the discrete criteria of subgroups of $\mathrm{PSL}(2,\mathbb{C}_{p})$. The method that they used are also derived from the Kleinian group. We not only follow the philosophy of the Kleinian group to study the $p-$adic M\"obius maps, but also we lay emphasis on the study of the difference between them. First, we classify the $p-$adic M\"obius maps into four types which is a bit different than those in the Kleinian group, and we study the geometrical characterization of the $p-$adic M\"obius maps from the aspects of fixed points in $\mathbb{P}_{Ber}$ and the invariant axes. Furthermore, we can decompose a $p-$adic M\"obius map $g$ into two involutions $\alpha,\beta$(an elliptic element of order $2$), namely $g=\alpha\circ \beta$, and the structure of the fixed points of $\alpha,\beta$ can reflect the type and other properties of $g$.  This method can be used to prove that if a subgroup of $\mathrm{PSL}(2,\mathbb{C}_{p})$ contains elliptic elements only, then all elements fix the same point in $\mathbb{H}_{Ber}$ without using the famous theorem--Cartan fixed point theorem, and this means that this subgroup has potentially good reduction(see concrete definitions in section 2). In the proof of the results, we should face three difficulties which do not exist in the archimedean settings. One is that  $\mathbb{P}^{1}(\mathbb{C}_{p})$ and $\mathbb{P}_{Ber}$ are not locally compact, another is that there exists a new kind of $p-$adic M\"obius maps--the wild elliptic elements whose geometrical structure of the fixed points are complicated, and the other is that the prime number $p$ affects the structure of the fixed points of $p-$adic M\"obius maps. We give a series of tables to compare the properties of $p-$adic M\"obius maps and those of the Kleinian group.
$\\*$

\begin{tabular}{|l|l|l|}
\hline
& the Kleinian group  &  $p-$adic M\"obius maps\\\hline
type&\begin{tabular}{l} loxodromic\ \ \ \ \ \ \ \ \ \ \ \ \\ \hline parabolic \\ \hline elliptic\\ \hline \ \ \ \ \end{tabular}& \begin{tabular}{l} loxodromic\ \ \ \ \ \ \ \ \ \ \ \ \\ \hline parabolic \\ \hline tame elliptic\\ \hline wild elliptic\end{tabular}\\
\hline
\end{tabular}
$\\*$

Assuming $g(z)=\frac{az+b}{cz+d}=\alpha\circ\beta$ with $ad-bc=1$, let $F_{g}, F_{\alpha}, F_{\beta}$ denote the fixed points of $g, \alpha, \beta$ in $\mathbb{P}^{1}\cup \mathbb{H}^{3}$ respectively.

$\\*$

\begin{tabular}{|l|l|l|}
\hline
 the Kleinian group & $F_{g}\subset \mathbb{P}^{1}\cup \mathbb{H}^{3}$ & $ F_{\alpha}\cap F_{\beta}$\\\hline
\begin{tabular}{l} loxodromic\ \ \ \ \ \ \ \ \ \ \ \ \\ \hline parabolic \\ \hline elliptic  \end{tabular}& \begin{tabular}{l} two points in $\mathbb{P}^{1}$ \\ \hline one point in $\mathbb{P}^{1}$ \\ \hline a geodesic line in $\mathbb{P}^{1}\cup \mathbb{H}^{3}$ \ \ \ \ \end{tabular}&\begin{tabular}{l} $\emptyset$  \\ \hline unique point in $\mathbb{P}^{1}$  \\ \hline unique point in $\mathbb{H}^{3}$  \end{tabular}\\
\hline
\end{tabular}
 \\
$\\*$

Let $\Int(A)$ denote the interior of the set $A$. Assuming $g(z)=\frac{az+b}{cz+d}=\alpha\circ\beta$ with $ad-bc=1$, let $F_{g}, F_{\alpha}, F_{\beta}$ denote the fixed points of $g, \alpha, \beta$  in $\mathbb{P}^{1}_{Ber}$ respectively.
$\\*$

\begin{tabular}{|l|l|l|}
\hline
 the $p-$adic M\"obius map& $F_{g}\subset \mathbb{P}^{1}\cup \mathbb{H}_{Ber}$ & $ F_{\alpha}\cap F_{\beta}$\\\hline
 loxodromic &  two points in $\mathbb{P}^{1}$ & $\emptyset$\\\hline
 parabolic  & $F_{g}\cap \mathbb{H}_{Ber}\neq \emptyset$ &  $\Int(F_{\alpha}\cap F_{\beta})\neq \emptyset$ \\\hline
tame elliptic & a geodesic line in $\mathbb{P}_{Ber}^{1}$ & unique point in $\mathbb{H}_{Ber}$\\\hline
wild elliptic & $\Int(F_{g})\neq \emptyset$  &  $\Int(F_{\alpha}\cap F_{\beta})\neq \emptyset$ \\
\hline
\end{tabular}
$\\*$

%\begin{thm} Let $g$ be a loxodromic or an elliptic element in $\mathrm{PSL}(2,\mathbb{C}_{p})$. Then there exists two involution $f,h\in \mathrm{PSL}(2,\mathbb{C}_{p})$ such that $g=f\circ h$. The axis $A$ of $h$ and the axis $B$ of $f$ are orthogonal to the axis $A_{g}$ of $g$, and the endpoints of $A$ are different from the endpoints of $B$.  Furthermore, when $p\ge 3$,  $A$ and $B$ do not intersect if and only if $g$ is a loxodromic element; $A$ and $B$ intersect if and only if $g$ is a elliptic element. When $p=2$, $g$ is a loxodromic element if and only if the two tailed geodesic lines $A_{x}\cap B_{y}=\emptyset$. $g$ is an elliptic element if and only if the two tailed geodesic lines $A_{x}\cap B_{y}\neq \emptyset$.
%\end{thm}

\begin{thm}\label{thm:reduction}If the subgroup $G\subset \mathrm{PSL}(2,\mathbb{C}_{p})$ contains elliptic elements only, then all the elements of $G$ share at least one fixed point in $\mathbb{H}_{Ber}$. Furthermore, $G$ has potentially good reduction, and $G$ is equicontinuous on $\mathbb{P}^{1}(\mathbb{C}_{p})$.
\end{thm}

A point $a\in \mathbb{P}^{1}(\mathbb{C}_{p})$ is called the limit point of a subgroup $G$ of $\mathrm{PSL}(2,\mathbb{C}_{p})$ if there exists a point $b\in\mathbb{P}^{1}(\mathbb{C}_{p})$ and an infinite sequence $\{g_{n}|n\ge 1\}\subset G$, where $g_{n}\neq g_{m}$ if $n\neq m$ with $\lim g_{n}(b)=a$. The set consisting of all limit points is called the {\it limit set}. $G$ is said to act discontinuously at $x\in \mathbb{P}^{1}(\mathbb{C}_{p})$ if there is a neighborhood $U$ of $x$ such that $g(U)\cap U=\emptyset$ for all but finitely many $g\in G$. The set of points at which $G$ acts discontinuously is called the {\it discontinuous set}.

\begin{thm}\label{thm:limitreduction} If the limit set of $G$ is empty, then $G$ has potentially good reduction.

\end{thm}

Two points $\alpha,\beta$ are called {\it antipodal} points if there exists an element $u\in \mathrm{PSL}(2,\mathcal{O}_{p})$ such that $u(0)=\alpha,u(\infty)=\beta$.

\begin{thm}\label{thm:5} For any $p$-adic M\"obius map $g$, there exists an element $u\in \mathrm{PSL}(2,\mathcal{O}_{p})$ such that $g=uf$, where either $f$ is a loxodromic element with antipodal fixed points, or $f=I$.

\end{thm}

In the last part, we extend the inequalities obtained by Gehring and Martin\cite{F.G1,F.G2},  Beardon and Short \cite{AI} to the non-archimedean settings.  These inequalities of $p$-adic M\"obius maps are between the matrix, chordal, three-point and unitary norms. This part of work can be applied to study the convergence of the sequence of $p-$adic M\"obius maps which can be viewed as a special cases of the work in \cite{CJE} and the discrete criteria of the subgroups of $\mathrm{PSL}(2,\mathbb{C}_{p})$.

For any two $p$-adic M\"obius maps $g,h$, we define the { \bf uniformly convergent  metric} on the $\mathrm{PSL}(2,\mathbb{P}^{1}(\mathbb{C}_{p}))$ as follows $$\rho_{0}(g,h)=\displaystyle\mathop{\sup}_{z\in \mathbb{P}^{1}(\mathbb{C}_{p})}(g(z),h(z)).$$ Let $M(g)=\parallel g-g^{-1}\parallel/\parallel g\parallel.$

\begin{thm}\label{thm:2}
Let $g$ be a $p$-adic M\"obius map.

{\rm (1)} If $p=3$, then $\rho_{0}(g,I)=M(g)$.

{\rm (2)} If $p=2$, then $2^{-1}M(g)\le \rho_{0}(g,I)\le 2 M(g)$.

\end{thm}

Let $\varepsilon(g)=\max\{\rho_{v}(g(z_{0}),z_{0}),\rho_{v}(g(z_{1}),z_{1}),\rho_{v}(g(z_{2}),z_{2})\}$, where $z_{0},z_{1},z_{2}$ are three distinct roots of the equation $z^{3}=1$.

\begin{thm}\label{thm:3}For any $p$-adic M\"obius map $g$, we have $2^{-1}\varepsilon(g)\le M(g)\le 6\varepsilon(g)$.

\end{thm}

Let $\varepsilon_{1}(g)=\{\rho_{v}(g(0),0),\rho_{v}(g(1),1),\rho_{v}(g(\infty),\infty)\}$.

\begin{thm}\label{thm:4}For any $p$-adic M\"obius map, $2^{-1}\varepsilon_{1}(g)\le M(g)\le \varepsilon_{1}(g)$.

\end{thm}

If $g$ is a parabolic element, we can improve the inequality. Let $\varepsilon_{2}(g)=\max\{\rho_{v}(g(0),0),\rho_{v}(g(\infty),\infty)\}$.

\begin{cor}\label{cor:1}If $g$ is a parabolic element, then $2^{-1}\varepsilon_{2}(g)\le M(g)\le \varepsilon_{2}(g)$.

\end{cor}

Let $\mathcal{U}=\mathrm{PSL}(2,\mathcal{O}_{p})$. We define $d(g,\mathcal{U})=\inf\{\rho_{0}(g,u)|u\in \mathcal{U}\}$. $d(g,\mathcal{U})$ measures how far an element from the group $\mathcal{U}$. This result is quite different from that in the archimedean setting.

\begin{thm}\label{thm:6}For any $p$-adic M\"obius map, either $d(g,\mathcal{U})=0$, if $g\in \mathcal{U}$, or $d(g,\mathcal{U})=1$, if $g\notin \mathcal{U}$.

\end{thm}

As an application of this result, we derive a discrete criteria of subgroups of $\mathrm{PSL}(2,\mathbb{C}_{p})$.

\begin{thm}\label{thm:7} If $G$ is a subgroup of $\mathrm{PSL}(2,\mathbb{C}_{p})$ with $G\cap \mathcal{U}=\{I\}$, then $G$ is a discrete subgroup.

\end{thm}

\begin{cor}\label{cor:2}If a subgroup $G\subset \mathrm{PSL}(2,\mathbb{C}_{p})$ contains unit element or loxodromic element only, then $G$ is a discrete subgroup.

\end{cor}

The other application of these inequalities is to get the convergence theorem of $p$-adic M\"obius maps.
\begin{thm}\label{thm:8}Let $\{g_{n}\}$ be a sequence of $p$-adic M\"obius maps, $z_{j},\; (j=1,2,3)$ be three distinct points and $g_{n}(z_{j})\rightarrow w_{j}$, where $w_{j}$ are also three distinct points. Then the sequence $\{g_{n}\}$ converges to a $p$-adic M\"obius map $g$ uniformly, where $g(z_{j})=w_{j},\; (j=1,2,3)$.

\end{thm}
\subsection{Motivation}
 Firstly the study of the Kleinian group in the archimedean case has been well developed for a rather long time. It is natural for one to consider a parallel theory in the non-archimedean settings. Here we wish to give a fairly clear picture of the $p-$adic M\"obius maps from the point of view of the Kleinian group. We study three aspects of the $p-$adic M\"obius maps. One is the group $\mathrm{PSL}(2;\mathcal{O}_p)$, another is the geometrical characterization of the $p-$adic M\"obius maps and its application, and the other is different norms of the $p-$adic M\"obius maps, since $\mathrm{PSL}(2;\mathcal{O}_p)$ is similar to the unitary group in the Kleinian group, the geometrical characterization of the $\mathrm{PSL}(2;\mathcal{O}_p)$ maps is useful in the study of the structure and dynamics of the subgroups of $\mathrm{PSL}(2;\mathbb{C}_p)$, and the study of norms of M\"obius maps is a very important topic, and many mathematicians such as Gehring and Martin \cite{F.G1,F.G2},  Beardon and Short \cite{AI} do a lot of works in this topic.

The other reason is that the three aspects can be viewed as tools to study other related topics. The properties of the norms of $p-$adic M\"obius maps have a lot of applications in three topics. One is the discrete criteria of subgroups of $\mathrm{PSL}(2;\mathbb{C}_p)$ and another is the pointwsie convergence of $p-$adic M\"obius maps(see \cite{WY,K.F,Q2}) and the other is the $p-$adic continued fraction. The geometrical characterization shows that the subgroup G containing elliptic element only shares one unique point, which means that a non-elementary group should have an loxodromic element which is very useful to study the dynamics of the discrete subgroup of $\mathrm{PSL}(2;\mathbb{C}_p)$. The rapid development of the Berkovich space and the arithmetical dynamical system (see \cite{A3},\cite{AK},\cite{ACAD},\cite{B2},\cite{B3},\cite{AC},\cite{AD},\cite{FLYD},\cite{FLWZ}, \cite{Q1})also promote studying the $p-$adic Mobius maps. Especially, the study of $p-$adic Mobius maps can be applied in studying the quantum mechanics and quantum cosmology \cite{D1,D2}). In this paper we give the affirrmative answers to all these three questions in the non-archimedean settings.

\subsection{Outline of the paper}
Outline of the paper. In section 1, we present our main results of the paper. In section 2, the basic theories of the $p-$adic analysis and the Berkovich space are briefly reviewed. In section 3, we obtain a few preliminary results of $p-$adic M\"obius maps. Section 4 contains proofs of the equations of $p-$adic M\"obius maps in $\mathrm{PSL}(2;\mathcal{O}_{p})$. In section 5, we give the results of the geometrical characterization of the $p-$adic M\"obius maps. In
section 6, the inequalities of $p-$adic M\"obius maps between matrix, chordal, three points, hyperbolic, and unitary norms are derived. In section 7, we prove
the decomposition theorem of the $p-$adic M\"obius maps and discuss its application.

\section{Some Preliminary Results }

\subsection{The non-archimedean space $\mathbb{P}^{1}(\mathbb{C}_{p})$ }
Let $p\ge 2$ be a prime number. Let $\mathbb{Q}_{p}$ be the field of $p$-adic
numbers and $\mathbb{C}_{p}$ be the completion of the algebraic closure of
$\mathbb{Q}_{p}$. Denote $|\mathbb{C}_{p}^{\ast}|$ the valuation group of $\mathbb{C}_{p}$. Then every element $r\in |\mathbb{K}^{\ast}|$ can be expressed as $r=p^{s}$ with $s\in \mathbb Q$.   We have the {\bf strong triangle inequality}
 \begin{displaymath}|x-y|\le \max\{|x|,|y|\}
 \end{displaymath} for $x,y\in \mathbb{C}_{p}$. If $x,y$ and $z$ are
points of $\mathbb{C}_{p}$ with $\vert x-y \vert<$  $\vert x-z\vert$, then $\vert x-z\vert=\vert y-z\vert$.

For any $a\in \mathbb{C}_{p}$, and $r>0$, we define
\begin{displaymath}
    D(a,r)^{-}=\{z\in\mathbb{C}_{p} \mid |z-a|< r\}
\end{displaymath}
to be the {\bf open} disk of radius $r$ and centered at $a$. Similarly,
\begin{displaymath}
    D(a,r)=\{z\in\mathbb{C}_{p} \mid|z-a|\le r\}
\end{displaymath}
denote the {\bf closed} disk of radius $r$ and centered at $a$.  Both $D(a,r)^{-}$ and $D(a,r)$ are closed and open topologically, and every point in disk $D(a,r)^{-}$ is the center. This denotes that if $x\in D(a,r)^{-}$, then $D(a,r)^{-}=D(x,r)^{-}$ (resp. $D(a,r)=D(x,r)$). By the strong triangle inequality, if two disks $D_{1}$ and $D_{2}$ in $\mathbb{C}_{p}$ have non-empty intersection, then either $D_{1}\subset D_{2}$, or $D_{2}\subset D_{1}$.

For any $z,w\in \mathbb{P}^{1}(\mathbb{C}_{p})$, we define the chordal distance

\begin{displaymath}
    \rho_{v}(z,w)=\frac{|z-w|}{\max\{1,|z|\}\max\{1,|w|\}}
\end{displaymath}
for $z,w\in \mathbb{C}_{p}$,
\begin{displaymath}
    \rho_{v}(z,w)=\frac{1}{\max\{1,|w|\}}
\end{displaymath}
for $w\in \mathbb{C}_{p}$ and $z=\infty$, and
\begin{displaymath}
    \rho_{v}(z,w)=0
\end{displaymath}
for $z=w=\infty$.

It follows the definition of the chordal distance that if $|z|\le1,|w|\le1$, then $\rho_{v}(z,w)=|z-w|$, and if $|z|>1,|w|\le1$, then by the strong triangle inequality,  we have $|z-w|=|z|$, and hence $\rho_{v}(z,w)=\frac{|z-w|}{|z|}=1$, and if $|z|>1,|w|>1$, then $\rho_{v}(z,w)=\frac{|z-w|}{|z||w|}=|\frac{1}{z}-\frac{1}{w}|$.

%\begin{lem}\label{yl:1} Let $f(z)=\sum^{\infty}_{i=0}c_{i}(z-a)^{i}\in \mathbb{C}_{p}[[z-a]]$ be a power series that converges on a disk $\bar{D}(a,r)$ of rational radius. Then $f$ is one-to-one on $\bar{D}(a,r)$ if and only if for all $i>1$, $|c_{i}|r^{i}<|c_{1}|r$. In this case, $|f'(z)|=|c_{1}|$ for all $z\in \bar{D}(a,r)$, and $\mathrm{diam}(f(\bar{D}(a,r)))=|c_{1}|r$, furthermore $|f(z)-f(w)|=|c_{1}||z-w|$ for any $z,w\in \bar{D}(a,r)$.
%\end{lem}

%For a power series convergent on a disk, it is known from \cite{RB} that

\begin{lem}[\cite{A.F}]\label{yl:1}Let $d_{0}>1$ be an integer which is not divisible by $p$, and $d=d_{0}p^{t}$ be a natural number. Let $\zeta$ be the primitive $d$-th root of unity. Then $|\zeta-1|=1$.
\end{lem}

\begin{lem}[ \cite{A.F}]\label{yl:2}
Let $\zeta$ be the primitive $p^{d}$-th root of unity. Then $|\zeta-1|=p^{-\frac{1}{p^{d-1}(p-1)}}$.
\end{lem}

\subsection{The Berckovich space}

We shall use a few properties of  about the structure of the Berkovich affine line and its topology. Here we give a brief introduction to the Berkovich space. More details can be found in \cite{BR}.

The underlying point set for the {\bf Berkovich affine line} $\mathbb{A}_{Ber}^{1}$ is the collection of all multiplicative seminorms $[\,\cdot\,]_{x}$ on the polynomial ring $\mathbb{C}_p[z]$ which extend the absolute value on $\mathbb{C}_{p}$. Recall that a {\bf multiplicative seminorms} on the ring $\mathbb{C}_p[z]$ is a function $[\,\cdot\,]_{x}:\mathbb{C}_{p}[z]\rightarrow [0, +\infty)$.

$\bullet\ [0]_{x}=0, [1]_{x}=1$;

$\bullet\ [fg]_{x}=[f]_{x}[g]_{x}$ for all $f,g\in \mathbb{C}_{p}[z]$;

$\bullet\ [f+g]_{x}\le \max\{[f]_{x}, [g]_{x}\}$ for all $f,g\in \mathbb{C}_{p}[z]$.

It is a {\bf norm} provided that $[f]_{x}=0$ if and only if $f=0$.  The topology on $\mathbb{A}_{Ber}^{1}$ is the weakest one for which the mapping $x\mapsto [f]_{x}$ is continuous for all $f\in \mathbb{C}_{p}[z]$.

%\begin{theorem} Every $x\in \mathbb{A}_{Berk}^{1}$ can be realized as $$[f]_{x}=\lim_{i\to \infty}[f]_{D(a_{i},r_{i})}$$ for some sequence of distinct of nested disks $\cdots\subset D(a_{2},r_{2}) \subset D(a_{1},r_{1}).$ If this sequence has a nonempty intersection, then either

%$(A)$ the intersection is a single point $a$, in which case $[f]_{x}=|f(a)|$, or

%$(B)$ the intersection is a closed disk $D(a,r)$, where $r$ belong to the value group of $\mathbb{C}_{p}$, in which case $[f]_{x}=[f]_{D(a,r)}$.

%\end{theorem}

Recall the Berkovich's classification Theorem: Every point $x\in \mathbb{A}_{Ber}^{1}$ can be viewed as a nested sequence of disks $D(a_{1},r_{1}) \supset D(a_{2},r_{2})\supset\cdots$. Moreover, points in $\mathbb A_{Ber}^1$ can be divided into
four types:

$\bullet$ A point in $\mathbb A_{Ber}^1$ corresponding to a nested sequence $\{D(a_{i},r_{i})\}$ of disks with $\lim{r_{i}}=0$ is said to be of {\bf type I}.

$\bullet$ A point in $\mathbb A_{Ber}^1$ corresponding to a nested sequence $\{D(a_{i},r_{i})\}$ of disks with nonempty intersection, for which $r=\lim r_{i}>0$ belongs to the valuation group $|\mathbb{C}^{\ast}|$ of $\mathbb{C}$, is said to be of {\bf type II}.

$\bullet$ A point in $\mathbb A_{Ber}^1$ corresponding to a nested sequence $\{D(a_{i},r_{i})\}$ of disks with nonempty intersection, for which $r=\lim r_{i}>0$ does not belong to the valuation group $|\mathbb{C}_{p}^{\ast}|$ of $\mathbb{C}_{p}$, is said to be of  {\bf type III}.

$\bullet$ A point in $\mathbb A_{Ber}^1$ corresponding to a nested sequence $\{D(a_{i},r_{i})\}$ of disks with empty intersection is said to be of {\bf type IV}.

We call points of type I, II, III the {\bf nonsingular points}, and points of type IV the {\bf singular points}. Every nonsingular point in $\mathbb A_{Ber}^1$ has a representation which is the intersection of the corresponding nested sequence of disks. So a nonsingular point in $\mathbb A_{Ber}^1$ can be identifying with a point $a$ (type I) or a disk $D(a,r)$ (type II, III).

We define a partial order on $\mathbb{A}_{Ber}^{1}$ as follows. For $x,y\in \mathbb{A}_{Ber}^{1}$, define $x\preceq y$ if and only if $[f]_{x}\le [f]_{y}$ for all $f\in \mathbb{C}_{p}[z]$. If $x,y$ are two points in $\mathbb A_{Ber}^1$ identifying with disks $D(a,r)$ and $D(a',r')$ respectively, then $x\preceq y$ if and only if $D(a,r)\subset D(a',r')$.

For a point $x \in \mathbb{A}_{Ber}^{1}$, we denote the set of elements larger than $x$ by
\begin{displaymath}
    [x,\infty[\,= \{w\in \mathbb{A}_{Ber}^{1}\mid x\preceq w\}.
\end{displaymath}
Observe that $[x,\infty[$ is isomorphic, as an ordered set, to $[0,+\infty[\,\subset \mathbb R$.

Given two points $x, y$ in $\mathbb{A}_{Ber}^{1}$, we have that
\begin{displaymath}
    [x,\infty[\,\cap\, [y,\infty[\,= [x\vee y,\infty[\,,
\end{displaymath}
where $x\vee y$ is the smallest element larger than $x$ and $y$. If $x$ is different from $y$, then the element $x\vee y$ is a point of type II. We also denote
\begin{displaymath}
    [x, y]= \{w \in \mathbb{A}_{Ber}^{1}\mid x\preceq w \preceq x\vee y\}\cup \{w \in \mathbb{A}_{Ber}^{1}\mid y\preceq w \preceq x\vee y\}.
\end{displaymath}
The sets $]x, y], [x, y[$ and $]x, y[$ are defined in the obvious way.

For a set $E\subset \mathbb K$, denote $\mathrm{diam}(E)=\sup_{z,w\in E}|z-w|$ the diameter of $E$ in the non-Archimedean metric. For $x\in\mathbb A_{Ber}^1$, which is corresponding to the nested sequence $\{D(a_i,r_i)\}$ of disks, the diameter of $x$ is given by
\begin{displaymath}
    \mathrm{diam}(x) = \displaystyle\mathop{\lim}_{i\to \infty} \mathrm{diam}(D(a_i,r_i)).
\end{displaymath}
For a nonsingular element $x\in\mathbb A_{Ber}^1$ identifying with the disk $D(a,r)$, the diameter of $x$ coincides with the diameter (radius $r$) of $D(a,r)$.

In order to endow the Berkovich affine line with a topology, we define an open
disk of $\mathbb{A}_{Ber}^{1}$ by
\begin{displaymath}
    \mathcal{D}(a,r)^{-} =\{x\in \mathbb{A}_{Ber}^{1} \mid \mathrm{diam}(a\vee x) < r\},
\end{displaymath}
for $a\in\mathbb K$ and $r>0$. Similaryly, a closed disk of $\mathbb{A}_{Ber}^{1}$ is defined by
\begin{displaymath}
    \mathcal{D}(a, r) = \{x\in \mathbb{A}_{Ber}^{1} \mid \mathrm{diam}(a\vee x)\le r\}.
\end{displaymath}

%\subsection{The Berkovich projective line}

Let $\mathbb P^1(\mathbb{C}_{p})$ be the projective line over $\mathbb{C}_{p}$, which can be viewed as $\mathbb P^1(\mathbb{C}_{p})=\mathbb{C}_{p}\cup\{\infty\}$. We can also introduce the {\bf Berkovich projective line} $\mathbb P_{Ber}^1$ over $\mathbb P^1(\mathbb K)$ similarly. In \cite{BR}, Baker and Rumely pointed out that $\mathbb{P}^{1}_{Ber}$ can be defined as $\mathbb{A}_{Ber}^{1}\cup \{\infty\}$, where $\infty\in\mathbb P^1(\mathbb K)$ is regarded as a point of type I. $\mathbb{P}^{1}_{Ber}$ can be also identifying with the disjoint union of a closed set $\mathcal{X}$ which is homeomorphic to $\mathcal{D}(0,1)$ and an open set $\mathcal{Y}$ which is homeomorphic to $\mathcal{D}(0,1)^{-}$. This provides a useful way to visualize  $\mathbb{P}^{1}_{Ber}$.

\begin{lem} [\cite{BR}]The space $\mathbb{P}^{1}_{Ber}$ is uniquely path-connected. More precisely, given any two distinct points $x,y\in \mathbb{P}^{1}_{Ber}$, there is a unique arc $[x,y]$ in $\mathbb{P}^{1}_{Ber}$ from $x$ to $y$, and if $x\in \mathcal{X}$ and $y\in \mathcal{Y}$, then the arc contains $\zeta_{\rm Gauss}\in\mathbb A_{Ber}^1$, where $\zeta_{\rm Gauss}$ is identifying with the disk $D(0,1)$.
\end{lem}

We say that a metric space $(X,d)$ is an {\bf $\mathbb{R}$-tree}, if for any two points $x,y\in X$, there is a unique arc from $x$ to $y$ and this arc is the geodesic segment.

\begin{lem}[\cite{BR}]Let $x,y\in \mathcal{D}(0,1)$. Then the metric $d(x,y)=2\mathrm{diam}(x\vee y)-\mathrm{diam}(x)-\mathrm{diam}(y)$ makes $\mathcal{D}(0,1)$ into an $\mathbb{R}$-tree.
\end{lem}

Therefore, we can define a metric on $\mathbb{P}^{1}_{Ber}$ as $d_{p}(x,y)=d(x,y)$ if $x,y\in \mathcal{X}$ which can be identified with $\mathcal{D}(0,1)$, and $d_{p}(x,y)=d(x,y)$ if $x,y\in \mathcal{Y}$ which can be identified with $\mathcal{D}(0,1)^{-}$, and $d_{p}(x,y)=d(x,\zeta_{Gauss})+d(\zeta_{Gauss},y)$ if $x\in \mathcal{X}$ and $y\in \mathcal{Y}$.

The {\bf Berkovich hyperbolic space} $\mathbb{H}_{Ber}$ is defined by \begin{displaymath}\mathbb{H}_{Ber}=\mathbb{P}^{1}_{Ber}\setminus\mathbb{P}^{1}(\mathbb{C}). \end{displaymath} Since $\infty\in \mathbb{P}^{1}_{Ber}$ is of type I,  $\mathbb{H}_{Ber}$ can be also viewed as $\mathbb{A}_{Ber}^{1}\setminus \mathbb{C}$. Thus $\mathbb{H}_{Ber}$ has a tree structure induced by $\mathbb{A}_{Ber}^{1}$. Over $\mathbb{H}_{Ber}$ we can define the {\bf hyperbolic distance}, \begin{displaymath}
    \rho(x,y)=2\log \mathrm{diam}(x\vee y)-\log\mathrm{diam}(x)-\log\mathrm{diam}(y),\quad x,y\in\mathbb H_{Ber}.
\end{displaymath}

\begin{lem}[\cite{BR}] $\mathbb{H}_{Ber}$ is a complete metric space under $\rho(x,y)$.
\end{lem}

\begin{lem}[\cite{BR}] Suppose that $w, y, x\in\mathbb{H}_{Ber}$. Then $\rho(x,y)=\rho(x,w)+\rho(w,y)$ if and only if $w$ belongs to $[x,y]$.
\end{lem}

\subsection{The action of a rational map $\phi$ over $\mathbb{P}_{Ber}^{1}$}
Let $\phi\in \mathbb{C}_{p}(T)$ be a nonconstant rational function of degree $d\ge 1$. Since type I points are dense in $\mathbb{P}_{Ber}^{1}$, for any $x\in \mathbb{H}_{Ber}$, there exists a sequence $x_{n}$ tending to $x$ with respect to the Berkovich topology. We can define $\phi(z)=\displaystyle\mathop{\lim}_{n\longrightarrow \infty}\phi(x_{n})$.

%We can choose a homogeneous lifting $F=(F_{1},F_{2})\in \mathbb{C}_{p}[X,Y]$ are homogeneous of degree $d$ and have no common zeros in $\mathbb{C}_{p}$, with $\phi=F_{1}/F_{2}$. For any seminorm $[[\cdot]]_{z}$ on $\mathbb{C}_{p}[X,Y]$, define $[[\cdot]]_{F(z)}$ by $$[[G]]_{F(z)}=[[G(F_{1}(X,Y),F_{2}(X,Y))]]_{z}$$ for $G\in \mathbb{C}_{p}[X,Y]$. Thus the action of $\phi$ on $\mathbb{P}_{Ber}^{1}$ is the one induced by the action of a homogeneous lifting $F=(F_{1},F_{2})$ on seminorms of $\mathbb{C}_{p}[X,Y]$, with $[[G]]_{F(z)}=[[G\circ F]]_{z}$.

If $d=1$, $\phi$ has an algebraic inverse and thus induces an automorphism of $\mathbb{P}_{Ber}^{1}$. Define $\mathrm{Aut}(\mathbb{P}_{Ber}^{1})$ to be the group of automorphisms of $\mathbb{P}_{Ber}^{1}$.  The following lemmas can be found in \cite{BR}.

%\begin{lem} $\mathrm{Aut}(\mathbb{P}_{Ber}^{1})\cong \mathrm{PSL}(2,\mathbb{C}_{p})$, the group of M$\ddot{o}$bius transformations (or linear fractional transformations) acting on $\mathbb{P}^{1}(\mathbb{C}_{p})$. In particular,

%$(A)$ Given triple of distinct points $(a_{0},a_{1},a_{\infty})$, $(\zeta_{0}, \zeta_{1}, \zeta_{\infty})\in \mathbb{P}^{1}(\mathbb{C}_{p})$, there is a unique $\phi\in \mathrm{Aut}(\mathbb{P}_{Ber}^{1})$ for which $\phi(a_{0})=\zeta_{0}$, $\phi(a_{1})=\zeta_{1}$, $\phi(a_{\infty})=\zeta_{\infty}$.

%$(B)$ Given a triple $(a_{0},A,a_{\infty})$ where $a_{0},a_{\infty}\in \mathbb{P}^{1}(\mathbb{C}_{p})$ are distinct points and $A$ is a type II point on the arc $[a_{0},a_{\infty}]$ and another triple $(\zeta_{0},Z,\zeta_{\infty})$ of the same kind, there is a $\phi\in \mathrm{Aut}(\mathbb{P}_{Ber}^{1})$ for which $\phi(a_{0})=\zeta_{0}$, $\phi(A)=Z$, and $\phi(a_{\infty})=\zeta_{\infty}$.

%\end{lem}

\begin{lem}[\cite{BR}]\label{yl:type} If $\phi(z)\in \mathbb{C}_{p}(z)$ is nonconstant, then $\phi: \mathbb{P}^{1}_{Ber}\longrightarrow \mathbb{P}^{1}_{Ber}$ takes points of each type $(I,II,III,IV)$ to points of the same type. Thus $\phi(z)$ has a given type if and only if $z$ does.

\end{lem}

%If $x\in \mathbb{A}_{Ber}^{1}$ which corresponds to a disk $D(a,r)$ is of type II, then $[f]_{x}=\max_{z\in D(a,r)}|f(z)|$ for any $f\in \mathbb{C}_{p}[T]$. If $f(T)$ has zeros $a_{1},a_{2},\ldots,a_{m}\in D(a,r)$, then by the Weierstrass Preparation Theorem $|f(z)|$ takes on its maximum value on $D(a,r)$ at each point of $D(a,r)\setminus\cup_{i=1}^{m}D(a_{i},r)^{-}$.  Thus for any $\phi=f/g\in \mathbb{C}_{p}(T)$, $[f/g]_{x}=\frac{[f]_{x}}{[g]_{x}}=\frac{\max_{z\in D(a,r)}|f(z)|}{\max_{z\in D(a,r)}|g(z)|}$. Then $[f/g]_{x}=\max_{z\in D(a,r)\setminus(\cup_{i=1}^{m}D(a_{i},r)^{-})\cup(\cup_{i=1}^{n}D(b_{i},r)^{-})}|f/g|$, where $a_{1},a_{2},\ldots,a_{m}\in D(a,r)$ are zeros of $f$ and $b_{1},b_{2},\ldots,b_{n}\in D(a,r)$ are zeros of $g$. It more concrete to understand the action of a rational map $\phi$ on  $x\in \mathbb{P}_{Ber}^{1}$.

\begin{lem}[\cite{BR}]\label{yl:image} Let $f(z)\in \mathbb{C}_{p}(z)$ be a nonconstant rational function, and suppose that $x\in \mathbb{A}_{Ber}^{1}$ is a point of type II, corresponding to a disc $D(a,r)$ in $\mathbb{C}_{p}$ under Berkovich's classification. Then $f(x)$ corresponds to the disc $D(b,R)$ if and only if there exist $a_{1},a_{2},\ldots , a_{m}, b_{1}, b_{2},\ldots, b_{n}\in \mathbb{C}_{p}$ for which $D(b,R)\setminus\cup_{i=1}^{n}D(b_{i},R)^{-}$ is the image under $f(z)$ of $D(a,r)\setminus \cup_{i=1}^{m}D(a_{i},r)^{-}$.

\end{lem}

\subsection{Reduction on rational function over $\mathbb{C}_{p}$}

Let $\mathcal{O}_{p}=\{z\in \mathbb{C}_{p}||z|\le 1\}$, $\mathcal{O}^{\ast}=\{z\in \mathbb{C}_{p}||z|=1\}$, $\mathcal{M}=\{z||z|<1\}$ and $k=\mathcal{O}_{p}/\mathcal{M}$. We also call $k$ the residue field of $\mathbb{C}_{p}$. If $x\in \mathcal{O}_{p}$, we denote the reduction of $x$ modulo $\mathcal{M}$ by  $\bar{x}$. For any $z\in \mathbb{C}_{p}$, there exists a homogeneous coordinate $[x,y]$ for $z$, where $x,y\in \mathcal{O}$ with at least one in $\mathcal{O}^{\ast}$. Reduction induces a well-defined map $\mathbb{P}^{1}(\mathbb{C}_{p})\rightarrow \mathbb{P}^{1}(k)$ by $\overline{[x,y]}=[\bar{x},\bar{y}]$. Any rational function $f(z)\in  \mathbb{C}_{p}(z)$ can be written in homogeneous coordinates as
$f([x, y])=[g(x, y), h(x, y)]$
where $g, h\in \mathcal{O}_{p}[x, y]$ are relatively prime homogeneous polynomials of
degree $d=\deg(f)$. We can ensure that at least one coefficient of either $g$ or
$h$ has valuation zero (i.e., absolute value 1). The reduction map induces a
map $\mathcal{O}_{p}[x, y]\rightarrow k[x, y]$.
\begin{defi}\label{df:reduction}Let $f(z)\in \mathbb{C}_{p}(z)$ be a map with homogenous presentation
$f([x, y])=[g(x, y), h(x, y)]$,
where $g, h \in  \mathcal{O}_{p}[x, y]$ are relatively prime homogeneous polynomials of
degree $d=\deg(f)$, and at least one coefficient of $g$ or $h$ has absolute value $1$.
We say that $f$ has good reduction if $\bar{g}$ and $\bar{h}$ have no common zeros in $k\times k$
besides $(x,y)=(0,0)$.
\end{defi}
If there is some linear fractional transformation $h\in \mathrm{PSL}(2,\mathbb{C}_{p})$ such that $h^{-1}\circ f\circ h$ has good reduction, we say that $f$ has {\bf potentially good reduction}.

\begin{lem}[\cite{SJ}]\label{yl:gdrdpM} Let $f\in \mathrm{PSL}(2, \mathbb{C}_{p})$ be a rational function of degree one. Then $f$ has good reduction if and only if $f\in \mathrm{PSL}(2, \mathcal{O})$.
\end{lem}

\begin{lem}[\cite{BR}]\label{yl:gdguass} Let $f\in \mathbb{C}_{p}(z)$ be a non-constant rational function. Then $f$ has good reduction if and only if $f^{-1}(\zeta_{Gauss})=\zeta_{Gauss}$.

\end{lem}

\section{The p-adic M\"obius maps}
We classify non-unit elements in PSL$(2,\mathbb{C}_{p})=\mathrm{SL}(2,\mathbb{C}_{p})/\{\pm I\}$. Since the product of all eigenvalues of $g\in $ PSL$(2,\mathbb{C}_{p})$ is one,  either the absolute value of each eigenvalue of $g$ is one or there exists at least one eigenvalue whose absolute value is larger than 1. Thus each non-unit element $g\in$ PSL$(2,\mathbb{C}_{p})$ falls into the following four classes:

(a) $g$ is said to be {\bf parabolic} if the absolute value of any eigenvalue of $g$ is 1, and $g$ can not be conjugated to a diagonal matrix.

(b) $g$ is said to be {\bf loxodromic} if there exists at least eigenvalue of $g$ whose absolute value is larger than $1$.

(c) $g\neq I$ is said to be {\bf elliptic} if the absolute value  of any eigenvalue of $g$ is $1$, and $g$ can be conjugated to a diagonal matrix.

In this paper, we classifies the elliptic elements more precisely.

(d) $g$ is said to be {\bf tame elliptic} if the two eigenvalues $\lambda_{1},\lambda_{2}$ of $g$ satisfy $|\lambda_{1}-1|=|\lambda_{2}-1|=1$.

(e) $g$ is said to be {\bf wild elliptic} if  one of the eigenvalues of $g$ lies in the disc $D(1,1)^{-}$.

For $g=(a_{ij})$ in the matrix ring $\mathrm{M}(m,\mathbb{C}_{p})$, the norm of $g$ is defined  by $\parallel g \parallel = \displaystyle\mathop{\max}_{1\le i\le m,1\le j\le m}\{|a_{ij}|\}$.  Obviously,  $\parallel g\parallel=0$ implies that each $a_{ij}=0$. It is easy to verify that $\parallel \alpha g\parallel=|\alpha|\parallel g\parallel$, $\parallel g+h\parallel\le \max\{\parallel g\parallel, \parallel h\parallel\}$ and $\parallel gh\parallel\le \parallel g\parallel\parallel h\parallel$.

For any element $g\in \mathrm{PSL}(2,\mathbb{C}_{p})$, there exist two lifts $g_{1},g_{2}$ in $\mathrm{SL}(2,\mathbb{C}_{p})$ with $\parallel g_{1}\parallel=\parallel g_{2}\parallel$. We define $\parallel g\parallel=\parallel g_{1}\parallel=\parallel g_{2}\parallel$. If $g,h\in \mathrm{PSL}(2,\mathbb{C}_{p})$ correspond the lifts $g_{1},g_{2}\in \mathrm{SL}(2,\mathbb{C}_{p})$ and $h_{1},h_{2}\in \mathrm{SL}(2,\mathbb{C}_{p})$  respectively, then we define $\parallel g-h\parallel=\displaystyle\mathop{\inf}_{1\le i\le2, 1\le j\le 2}\parallel g_{i}-h_{j}\parallel$.

%We can choose a homogeneous lifting $F=(F_{1},F_{2})\in \mathbb{C}_{p}[X,Y]$ are homogeneous of degree $d$ and have no common zeros in $\mathbb{C}_{p}$, with $\phi=F_{1}/F_{2}$. For any seminorm $[[\cdot]]_{z}$ on $\mathbb{C}_{p}[X,Y]$, define $[[\cdot]]_{F(z)}$ by $$[[G]]_{F(z)}=[[G(F_{1}(X,Y),F_{2}(X,Y))]]_{z}$$ for $G\in \mathbb{C}_{p}[X,Y]$. Thus the action of $\phi$ on $\mathbb{P}_{Berk}^{1}$ is the one induced by the action of a homogeneous lifting $F=(F_{1},F_{2})$ on seminorms of $\mathbb{C}_{p}[X,Y]$, with $[[G]]_{F(z)}=[[G\circ F]]_{z}$.

If $d=1$, $\phi$ has an algebraic inverse and thus induces an automorphism of $\mathbb{P}_{Ber}^{1}$. Define $\mathrm{Aut}(\mathbb{P}_{Ber}^{1})$ to be the group of automorphisms of $\mathbb{P}_{Ber}^{1}$.  The following lemmas can be found in \cite{BR}.

\begin{lem}[\cite{BR}]\label{yl:iso}The path distance metric $\rho(x,y)$ on $\mathbb{H}_{Ber}$ is independent of the choice of homogenous coordinates on $\mathbb{P}^{1}_{Ber}$, in the sense that if $h(z)\in \mathbb{C}_{p}(z)$ is a $p$-adic M\"obius map, then $\rho(h(x),h(y))=\rho(x,y)$ for all $x,y\in \mathbb{H}_{Ber}$.

\end{lem}

\begin{lem}[\cite{QY}]\label{yl:3.1} Let $w$ satisfy $|w|=\lambda|w-a|$, where $\lambda\in |\mathbb{C}_{p}^{\ast}|$, and $w,a\in \mathbb{C}_{p}$.

$(1)$ If $\lambda>1$, then $w\in D(a,\frac{|a|}{\lambda})\setminus D(a,\frac{|a|}{\lambda})^{-}$.

$(2)$ If $\lambda=1$, then $w\in \mathbb{P}^{1}(\mathbb{C}_{p})\setminus D(0,|a|)^{-}\cup D(a,|a|)^{-}$.

$(3)$ If $0<\lambda<1$, then $w\in D(0,\lambda|a|)\setminus D(0,\lambda|a|)^{-}$.

\end{lem}

\begin{pro} \label{pro:conj}A $p$-adic M\"obius map $g$ is

$(1)$ {\it parabolic} if it is conjugate to $z\rightarrow z+1$;

$(2)$ {\it elliptic} if it is conjugate to $z\rightarrow kz$ for some $k$ with $|k|=1$, $k\neq 1$;

$(3)$ {\it loxodromic} if it is conjugate to $z\rightarrow kz$ for some $k$ with $|k|>1$.

\end{pro}

%\begin{proof}If $f$ is an elliptic element or a loxodromic element,  there exist two fixed points $a_{f}, r_{f}\in \mathbb{P}^{1}(\mathbb{C}_{p})$ of $f$. We can choose a $p$-adic M\"obius map $h$ such that $h(a_{f})=0,h(r_{f})=\infty$. Then $hfh^{-1}$ fixes the $0$ and $\infty$ which yields that there exists a $p$-adic number $k$ such that $hfh^{-1}(z)=kz$. If $f$ is a loxodromic element, then one of the absolute values of the eigenvalues of $f$ is larger than one. Since conjugation does not change the trace of the matrix, we have either $|k|<1$ or $|k|>1$. If $|k|<1$,  let $\sigma=1/z$ and we have $\sigma hfh^{-1}\sigma (z)=(1/k)z$. If $f$ is elliptic, we have $|k|=1$ and $k\neq 1$, otherwise $f=z$.

%If $f$ is a parabolic element, by the definition, there exists a $p$-adic M\"obius map $h$ such that we have $hfh^{-1}=z+k$. Let $\sigma(z)=z/k$. Then $\sigma hfh^{-1}\sigma^{-1}=z+1$. We draw our conclusion.

%\end{proof}

\begin{pro} Let $f$ and $g$ be two $p$-adic M\"obius maps neither of which is the identity.
Then $f$ and $g$ are conjugate  if and only if $trace(f) = trace(g)$.
\end{pro}

%\begin{proof}

%By Proposition \ref{pro:conj}, if $tr(f)=tr(g)$, then $f$ and $g$ are conjugate. It is clear that the converse is also true.

%\end{proof}

\begin{pro}\label{pro:tr}Let $g\neq I$ be a $p$-adic M\"obius map. Then

$(1)$ $g$ is parabolic if and only if $(a+d)^{2}-4=0$;

$(2)$ $g$ is elliptic if and only if $0<|(a+d)^{2}-4|\le1$;

$(3)$ $g$ is loxodromic if and only if $|(a+d)^{2}-4|>1$.
\end{pro}

\begin{proof} Since the trace of the matrix is invariant under the conjugation, without lose of generality, we can assume that $g(z)=\frac{a}{d}z$, if $g$ is elliptic or loxodromic which yields that $tr(g)=a+d=\lambda+1/\lambda$, where $\lambda$ is the eigenvalue of $g$. Thus $(a+d)^{2}-4=(\lambda-1/\lambda)^{2}$. By the non-archimedean property of the metric, we have $|\lambda-1/\lambda|>1$, if $g$ is loxodromic, and $0<|\lambda-1/\lambda|<1$, if $g$ is elliptic. If $g$ is parabolic, we can assume that $g(z)=z+1$  which yields that $a+d=2$.

Conversely, if $(a+d)^{2}-4=0$, then $\lambda-1/\lambda=0$ which denotes $\lambda^{2}=1$. Since $g\in {\rm PSL}(2,\mathbb{C}_{p})$, we have $\lambda=1$ which yields that $g$ is parabolic. If $|(a+d)^{2}-4|>1$, then $|\lambda-1/\lambda|>1$ which denotes that $|\lambda|>1$ or $|1/\lambda|>1$. Thus $g$ is loxodromic. If $0<|(a+d)^{2}-4|<1$, then $0<|\lambda-1/\lambda|<1$ which yields $|\lambda|<1$. Thus $g$ is elliptic.

\end{proof}

\section{The properties of $\mathrm{PSL}(2,\mathcal{O}_{p})$}

\begin{lem}[\cite{SJ}] If $f\in {\rm PSL}(2,\mathcal{O}_{p})$, we have $\rho_{v}(f(z),f(w))=\rho_{v}(z,w)$.

\end{lem}

\begin{lem}[\cite{QY}] \label{yl:main} Let $g$ be any $p$-adic M\"obius map. Then $$\rho(g(\zeta_{Gauss}),\zeta_{Gauss})=2\log_{p}\|g\|.$$

\end{lem}

\begin{lem}\label{lem:lip1} Let $g$ be any $p$-adic M\"obius map. Then the best Lipschitz
constant (relative the chordal metric) for $g$ is given by $L(g)=p^{\rho(\zeta_{Gauss},g(\zeta_{Gauss}))}$, namely $\rho_{v}(g(z),g(w))\le L(g)\rho_{v}(z,w)$. Furthermore, there exist at least two points $z,w\in \mathbb{P}^{1}(\mathbb{C}_{p})$ such that $\rho_{v}(g(z),g(w))= L(g)\rho_{v}(z,w)$.
\end{lem}

\begin{proof} We can assume that $g(z)=\frac{az+b}{cz+d}$. The element $g$ has at least one fixed point $a_{g}$. If  $|a_{g}|\le 1$, let $h(z)=z-a_{g}$ and  $\iota(z)=1/z$, and then  $\iota hgh^{-1}\iota^{-1}$ fixes $\infty$. Since $h$ and $\iota$ fix the point $\zeta_{Gauss}$, we have $$\rho(\iota hgh^{-1}\iota^{-1}(\zeta_{Gauss}),\zeta_{Gauss})=\rho(\iota hg(\zeta_{Gauss}),\zeta_{Gauss})=\rho(g(\zeta_{Gauss}),\zeta_{Gauss})$$ which yields that $L(g)=L(\iota hgh^{-1}\iota^{-1})$. If $|a_{g}|>1$, let $h(z)=z-1/a_{g}$ and $\iota(z)=1/z$. Thus  $\iota h\iota g\iota h^{-1}\iota^{-1}$ fixes $\infty$. Similarly, since $h$ and $\iota$ fix the point $\zeta_{Gauss}$, we have $$\rho(\iota h\iota g\iota h^{-1}\iota^{-1}(\zeta_{Gauss}),\zeta_{Gauss})=\rho(\iota h\iota g(\zeta_{Gauss}),\zeta_{Gauss})=\rho(g(\zeta_{Gauss}),\zeta_{Gauss})$$ which yields that $L(g)=L(\iota h\iota g\iota h^{-1}\iota^{-1})$. Therefore, without loss of generality, we can assume that $g(z)=\frac{az+b}{d}$ with $ad=1$.  If $|a|>1$,  we consider the inverse $g^{-1}(z)=\frac{d}{a}z-\frac{b}{a}$, since $L(g)=L(g^{-1})$. Thus we can assume that $|a|\le 1$.

For any $z,w\in \mathbb{P}^{1}(\mathbb{C}_{p})$, we have

$$\rho_{v}(g(z),g(w))=\frac{|g(z)-g(w)|}{\max\{1,|g(z)|\}\max\{1,|g(w)|\}}\quad \quad \quad \quad\quad \quad \quad \quad$$

$$\quad \quad  =\frac{|z-w|}{\max\{|az+b|,|d|\}\max\{|aw+b|,|d|\}}$$

$$ \quad\quad \quad\quad \quad\quad \quad \quad \quad=\frac{|z-w|}{\max\{1,|z|\}\max\{1,|w|\}}\frac{\max\{1,|z|\}\max\{1,|w|\}}{\max\{|az+b|,|d|\}\max\{|aw+b|,|d|\}}$$

$$\quad\quad \quad \quad \quad=\rho_{v}(z,w)\frac{\max\{1,|z|\}\max\{1,|w|\}}{\max\{|az+b|,|d|\}\max\{|aw+b|,|d|\}}$$

$$\quad\quad \quad \quad \quad=\rho_{v}(z,w)\frac{|a|^{2}\max\{1,|z|\}\max\{1,|w|\}}{\max\{|a^{2}z+ab|,1\}\max\{|a^{2}w+ab|,1\}}.$$

When $|a|=1$, let $D_{1}=D(-\frac{b}{a},1)$ and $D_{2}=D(0,1)$.  If $D_{1}=D_{2}$, then $|\frac{b}{a}|\le 1$ which yields that $|b|\le |a|=1$, namely $L(g)=1$. If $z,w\in D_{1}=D_{2}$, then $$\frac{|a|^{2}\max\{1,|z|\}\max\{1,|w|\}}{\max\{|a^{2}z+ab|,1\}\max\{|a^{2}w+ab|,1\}}=1,$$ namely $\rho_{v}(g(z),g(w))=L(g)\rho_{v}(z,w)$. If $z\in D_{1}=D_{2}$ and $w\notin D_{1}=D_{2}$, $$\frac{|a|^{2}\max\{1,|z|\}\max\{1,|w|\}}{\max\{|a^{2}z+ab|,1\}\max\{|a^{2}w+ab|,1\}}=\frac{|w|}{|w+b/a|}.$$ Since $|\frac{b}{a}|\le 1$, we have $|w|=|w+\frac{b}{a}|$ which yields that $\frac{|w|}{|w+b/a|}=1$. Thus $\rho_{v}(g(z),g(w))=L(g)\rho_{v}(z,w)$. If $z,w\notin D_{1}=D_{2}$, we have $$\frac{|a|^{2}\max\{1,|z|\}\max\{1,|w|\}}{\max\{|a^{2}z+ab|,1\}\max\{|a^{2}w+ab|,1\}}=\frac{|w||z|}{|z+b/a||w+b/a|}.$$ Since $|\frac{b}{a}|\le 1$, we have $|w|=|w+\frac{b}{a}|$ and $|z|=|z+\frac{b}{a}|$ which yields that $$\frac{|z||w|}{|z+b/a||w+b/a|}=1.$$ Thus $\rho_{v}(g(z),g(w))=L(g)\rho_{v}(z,w)$.  If $D_{1}\cap D_{2}=\emptyset$, then $|\frac{b}{a}|>1$ which yields that $|b|>|a|=1$, namely $L(g)=|b|^{2}$. If $z,w\in D_{1}$, then $|z|>1,|w|>1,|z+b/a|\le 1, |w+b/a|\le 1$ which yields that $$\frac{|a|^{2}\max\{1,|z|\}\max\{1,|w|\}}{\max\{|a^{2}z+ab|,1\}\max\{|a^{2}w+ab|,1\}}=|\frac{b}{a}|^{2}=|b|^{2}.$$ Thus $\rho_{v}(g(z),g(w))=L(g)\rho_{v}(z,w)$. If $z,w\in D_{2}$, then $|z|\le 1,|w|\le 1,|z+b/a|< 1, |w+b/a|< 1$ which yield that $$\frac{|a|^{2}\max\{1,|z|\}\max\{1,|w|\}}{\max\{|a^{2}z+ab|,1\}\max\{|a^{2}w+ab|,1\}}=\frac{1}{|z+b/a||w+b/a|}=|\frac{a}{b}|^{2}=|b|^{-2}\le L(g).$$ Thus $\rho_{v}(g(z),g(w))\le L(g)\rho_{v}(z,w)$. If $z\in D_{1}$ and $w\in D_{2}$, then $|z|>1,|w|\le 1,|z+b/a|\le 1, |w+b/a|> 1$ which yields that $$\frac{|a|^{2}\max\{1,|z|\}\max\{1,|w|\}}{\max\{|a^{2}z+ab|,1\}\max\{|a^{2}w+ab|,1\}}=\frac{|z|}{|w+b/a|}=1\le L(g).$$  If $z\notin D_{1}\cup D_{2}$, then $|z|>1,|z+b/a|> 1$ which yields that $$\frac{\max\{1,|z|\}}{\max\{|a^{2}z+ab|,1\}}=\frac{|z|}{|z+b/a|}=1.$$ Thus if $z,w\notin D_{1}\cup D_{2}$, then $\rho_{v}(g(z),g(w))=\rho_{v}(z,w)\le L(g)\rho_{v}(z,w)$. If $z\notin D_{1}\cup D_{2}, w\in D_{1}$, then $\rho_{v}(g(z),g(w))=\rho_{v}(z,w)|w|=|b|\rho_{v}(z,w)\le L(g)\rho_{v}(z,w)$. If  $z\notin D_{1}\cup D_{2}, w\in D_{2}$, then $\rho_{v}(g(z),g(w))=\rho_{v}(z,w)/|w+b/a|=|a/b|\rho_{v}(z,w)\le L(g)\rho_{v}(z,w)$. Therefore, $\rho_{v}(g(z),g(w))\le L(g)\rho_{v}(z,w)$.

When $|a|<1$, let $D_{1}=D(-\frac{b}{a},\frac{1}{|a|^{2}})$ and $D_{2}=D(0,1)$. If $D_{2}\subset D_{1}$, then $|\frac{b}{a}|\le \frac{1}{|a|^{2}}$ which yields that $L(g)=|d|^{2}=|a|^{-2}$.  If $z,w\in D_{1}$, then $$\frac{|a|^{2}\max\{1,|z|\}\max\{1,|w|\}}{\max\{|a^{2}z+ab|,1\}\max\{|a^{2}w+ab|,1\}}=|a|^{2}\max\{1,|z|\}\max\{1,|w|\}\le \frac{1}{|a|^{2}}=L(g),$$ namely $\rho_{v}(g(z),g(w))\le L(g)\rho_{v}(z,w)$. If $z,w\notin D_{1}$, then $$\frac{|a|^{2}\max\{1,|z|\}\max\{1,|w|\}}{\max\{|a^{2}z+ab|,1\}\max\{|a^{2}w+ab|,1\}}=\frac{|a|^{2}|z||w|}{|a|^{4}|z||w|}=|a|^{-2}$$ which yields that $\rho_{v}(g(z),g(w))=L(g)\rho_{v}(z,w)$. If $z\in D_{1}, w\notin D_{1}$, then $$\frac{|a|^{2}\max\{1,|z|\}\max\{1,|w|\}}{\max\{|a^{2}z+ab|,1\}\max\{|a^{2}w+ab|,1\}} =\frac{|a|^{2}|w|\max\{1,|z|\}}{|a|^{2}|w|}=\max\{1,|z|\}\le \frac{1}{|a|^{2}}=L(g).$$ Thus $\rho_{v}(g(z),g(w))\le L(g)\rho_{v}(z,w).$ If $D_{1}\cap D_{2}= \emptyset$, then $|\frac{b}{a}|>\frac{1}{|a|^{2}}$ which yields that $|b|>\frac{1}{|a|}$. Thus $L(g)=|b|^{2}$. If $z\in D_{2}$, then $$\frac{\max\{1,|z|\}}{\max\{1,|a^{2}z+ab|\}}\le 1.$$ If $z\in D_{1}$, then $$\frac{\max\{1,|z|\}}{\max\{1,|a^{2}z+ab|\}}=\frac{|b|}{|a|}.$$ If $z\notin D_{1}\cup D_{2}$, then $$\frac{\max\{1,|z|\}}{\max\{1,|a^{2}z+ab|\}}=\frac{|z|}{|a^{2}z+ab|}.$$ Thus if $z,w\in D_{1}$, then $$\frac{|a|^{2}\max\{1,|z|\}\max\{1,|w|\}}{\max\{1,|a^{2}z+ab|\}\max\{1,|a^{2}w+ab|\}}=|b|^{2}$$ which yields that $\rho_{v}(g(z),g(w))= L(g)\rho_{v}(z,w)$. If $z,w\in D_{2}$, then $$\frac{|a|^{2}\max\{1,|z|\}\max\{1,|w|\}}{\max\{1,|a^{2}z+ab|\}\max\{1,|a^{2}w+ab|\}}\le |a|^{2}\le 1\le L(g).$$ Thus $\rho_{v}(g(z),g(w))\le L(g)\rho_{v}(z,w)$. If $z_{1}\in D_{1},w\in D_{2}$, then $$\frac{|a|^{2}\max\{1,|z|\}\max\{1,|w|\}}{\max\{1,|a^{2}z+ab|\}\max\{1,|a^{2}w+ab|\}}\le |a||b|\le L(g).$$ Thus $\rho_{v}(g(z),g(w))\le L(g)\rho_{v}(z,w).$ If $z,w\notin D_{1}\cup D_{2}$, then $$\frac{|a|^{2}\max\{1,|z|\}\max\{1,|w|\}}{\max\{1,|a^{2}z+ab|\}\max\{1,|a^{2}w+ab|\}}=\frac{|a|^{2}|z||w|}{|a^{2}z+ab||a^{2}w+ab|}.$$ If $|z|=|z+\frac{b}{a}|$, then $\frac{|z|}{|z+b/a|}=1$. If $|z|<|z+b/a|$, then  $\frac{|z|}{|z+b/a|}<1$. If $|z|>|z+b/a|$, then $|z|=|b/a|$. Therefore, we have $$\frac{|a|^{2}|z||w|}{|a^{2}z+ab||a^{2}w+ab|}\le |b|^{2}.$$ Thus $\rho_{v}(g(z),g(w))\le L(g)\rho_{v}(z,w)$.

\end{proof}
\begin{rmk}
{\rm
In \cite{SJ}, the Lipschitz constant of the rational map with respect to the chordal metric is derived by the resultant(see definition below). Let $\mathcal{O}_{p}^{\ast}=\{\alpha\in \mathbb{C}_{p}: |\alpha|= 1\}$. Let $g(z)=a_{0}z^{n}+a_{1}z^{n-1}+...+a_{n-1}z+a_{n}$ and $h(z)=b_{0}z^{m}+b_{1}z^{n-1}+...+b_{m-1}z+b_{m}.$  The rational function $\phi(z)=\frac{g(z)}{h(z)}\in \mathbb{C}_{p}(z)$ is called a normalized form if all the coefficients of $g(z)$ and $h(z)$ are in $\mathcal{O}_{p}$ and at least one coefficient of $g(z)$ and $h(z)$ are in $\mathcal{O}_{p}^{\ast}$.
Let $\alpha_{i},1\le i\le n$ be the roots of $g(z)$, and $\beta_{i}, 1\le i\le m$ be the roots of $h(z)$. Let $\mathrm{Res}(g(z),h(z))=a_{0}^{n}b_{0}^{m}\prod_{i=1}^{n}\prod_{j=1}^{m}(\alpha_{i}-\beta_{j})$ be the resultant of two polynomials, and  $\mathrm{Res}(\phi)=\mathrm{Res}(g(z),h(z))$ be the resultant of the rational function $\phi$. The absolute value of the resultant $\mathrm{Res}(\phi)$  depends only on the map $\phi$.

\begin{lem}[\cite{SJ}]\label{yl:lip} Let $\phi:\mathbb{P}^{1}(\mathbb{C}_{p})\longrightarrow\mathbb{P}^{1}(\mathbb{C}_{p})$ be a rational map. Then $\rho_{v}(\phi(z),\phi(w))\le |\mathrm{Res(\phi)}|^{-2}\rho_{v}(z,w)$  for all $v,w\in \mathbb{P}^{1}(\mathbb{C}_{p}).$

\end{lem}

Let $g(z)=\frac{az+b}{cz+d}\in \mathrm{PSL}(2,\mathbb{C}_{p})$. By Lemma \ref{yl:lip}, we have
$$\mathrm{Res}(g(z))=\mathrm{Res}((a/t)z+b/t, (c/t)z+d/t)=1/\parallel g\parallel^{2}=L(g),$$ where $t=\max\{|a|,|b|,|c|,|d|\}=\parallel g\parallel$.

In \cite{SJ}, it is shown that there are two points $x,y\in \mathbb{P}^{1}(\mathbb{C}_{p})$ such that $\displaystyle\mathop{\sup}_{x\neq y}\rho_{v}(\phi(x),\phi(y))/\rho_{v}(x,y)=|\mathrm{Res(\phi)}|^{-2}$. Since $\mathbb{P}^{1}(\mathbb{C}_{p})$ is not compact, the supreme can not be omitted in general cases. However,  we get the Lipschitz constant by other method and show that we can get the supreme when it is the $p$-adic M\"obius map.
}
\end{rmk}

Proof of Theorem \ref{thm:1}

\begin{proof} Since $\rho(g(\zeta_{Gauss}),\zeta_{Gauss})=2\log_{p}\parallel g\parallel$, $\parallel g\parallel=1$ is equivalent to $\rho(g(\zeta_{Gauss}),\zeta_{Gauss})=0$ which yields that $L(g)=1$. The converse is also true. This means that $(1)$, $(2)$, $(3)$ are equivalent.

$(2)\Rightarrow (4)$
If $L(g)=1$,  $L(g^{-1})=1$. Hence by Lemma \ref{lem:lip1}, we have $\rho_{v}(z,w)=\rho_{v}(g^{-1}(g(z)),g^{-1}(g(w)))\le L(g^{-1})\rho_{v}(g(z),g(w))\le L(g)\rho_{v}(z,w)$ which yields that $\rho_{v}(g(z),g(w))=\rho_{v}(z,w)$.

$(4)\Rightarrow (2)$
If $g$ is a chordal isometry, then $\rho_{v}(g(z),g(w))=\rho_{v}(z,w)$. By Lemma \ref{lem:lip1}, we have $L(g)=1$.

$(1)\Rightarrow (5)$
Since $\parallel g\parallel =1$, let $g=\frac{az+b}{cz+d}$, we have $\max\{|a|,|b|,|c|,|d|\}=1$ which yields that $g\in \mathrm{PSL}(2,\mathcal{O}_{p})$.

$(5)\Rightarrow (1)$
If $g\in \mathrm{PSL}(2,\mathcal{O}_{p})$, let $g=\frac{az+b}{cz+d}$, then $\max\{|a|,|b|,|c|,|d|\}\le 1$ and $ad-bc=1$. If $\max\{|a|,|b|,|c|,|d|\}< 1$, then $|ad-bc|<\max\{|ad|,|bc|\}<1$. This is a contradiction. Hence $\max\{|a|,|b|,|c|,|d|\}=1$ which yields that $\parallel g\parallel=1$.

$(1)\Rightarrow (6)$
Since $\parallel g\parallel=1$, $ \parallel g^{-1}\parallel=1$. Since $\parallel gh\parallel\le \parallel g\parallel \parallel h\parallel\le \parallel h\parallel$, and $\parallel h\parallel=\parallel g^{-1}gh\parallel\le \parallel g^{-1}\parallel \parallel gh\parallel\le \parallel gh\parallel$,  we have $\parallel gh\parallel=\parallel h\parallel$. Similarly, $\parallel hg\parallel=\parallel h\parallel$. We can rewrite $h$ as $hg^{-1}$, and then we have $\parallel ghg^{-1}\parallel=\parallel hg^{-1}\parallel=\parallel h\parallel$, since $g\in \mathrm{PSL}(2,\mathcal{O}_{p})$.

$(6)\Rightarrow (1)$
Let $h\in \mathrm{PSL}(2,\mathcal{O}_{p})$. Then $\parallel h\parallel=\parallel gh\parallel=\parallel g\parallel$ which yields that $\parallel g\parallel=1$.

\end{proof}

\newcommand{\DF}[2]{\displaystyle\frac{#1}{#2}}

The metric properties of $\mathrm{PSL}(2,\mathcal{O}_{p})$ can be used to study the $p-$adic continued fractions. An {\it infinite $p-$adic continued fraction } is a formal expression

\begin{displaymath}
\DF{a_{1}}{b_{1}+\DF{a_{2}}{b_{2}+\DF{a_{3}}{b_{3}+\cdots}}},
\end{displaymath}
where $a_{i}, b_{i}\in \mathbb{C}_{p}$ and $a_{i}\neq 0$. We denote this continued fraction by $\mathbb{K}(a_{n}|b_{n})$. Let $t_{n}=\frac{a_{n}}{z+b_{n}}$ and $T_{n}=t_{1}\circ t_{2}\circ t_{3}\circ \cdots$ for $n=1,2,3,\ldots$. The continued fraction is said to be convergent classically if the sequence $\{T_{n}\}$ converges, else it is said to diverge classically. In the following part, we study the simplest case $\mathbb{K}(1|b_{i})$ with $|b_{i}|\le 1$.  Since  $T_{n+1}(\infty)=T_{n}(0)$, by Theorem \ref{thm:1},  if the sequence $\{T_{n}(0)\}$ converges, then $0=\displaystyle\mathop{\lim}_{n\rightarrow \infty}\rho_{v}(T_{n}(0),T_{n}(\infty))=\rho_{v}(0,\infty)=1$. This is a contradiction. This implies that the continued fraction $\mathbb{K}(1|b_{i})$ with $|b_{i}|\le 1$ diverge classically.
\begin{ex}
 \begin{displaymath}
\DF{1}{1+\DF{1}{1+\DF{1}{1+\cdots}}}
\end{displaymath} does not converge classically.
\end{ex}
This example shows that the convergence and divergence of $p-$adic continued fractions are different from those in complex settings.
\section{Reduction and $p-$adic M\"obius maps }

We call an elliptic element $f\in \mathrm{PSL}(2,\mathbb{C}_{p})$ of order $2$  an {\bf involution}. If $g\in \mathrm{PSL}(2,\mathbb{C}_{p})$ is a loxodromic element or an elliptic element,  $g$ have two fixed points $a_{g}$ and $r_{g}$ in $\mathbb{P}^{1}(\mathbb{C}_{p})$. We call the geodesic line $A_{g}$ which connects $a_{g}$ and $r_{g}$ the {\bf axis} of $g$. Let $A$ be a geodesic line in $\mathbb{P}_{Ber}^{1}$. A geodesic line $B$ is {\bf orthogonal} to $A$ if there exists a $p$-adic M\"obius transformation $f$ which is an involution such that endpoints of $B$  are two fixed points of $f$, and $f$ interchanges endpoints of $A$.

We introduce a new conception,  {\bf a tailed geodesic line}, in order to analyze the geometrical characterization of $2-$adic M\"obius maps. Let $A$ be a geodesic line and $x$ be the point satisfying $\inf_{y\in A}{\rho(x,y)}=\log_{p}2$. Choosing any point $y\in A$, $A_{x}=A\cup[x,y]$. Since the Berkovich space is a $\mathbb{R}-$tree, it is obviously that $A_{x}$ is independent of the chosen point $y$, $A$ is the geodesic line associated with $A_{x}$, and $x$ is called {\bf a tail}. If $A$ is the axis of an involution $f$, then there exists a unique point $x$ fixed by $f$ such that $A_{x}$ is a tailed geodesic line. Hence we call $A_{x}$ the tailed axis of $f$.

\begin{lem}\label{thm:1.1} Let $g$ be a $p-$adic M\"obius map.  Then there exist two involution $f,h\in \mathrm{PSL}(2,\mathbb{C}_{p})$ such that $g=f\circ h$. Furhtermore

{\rm (1)} The axes $A$ of $h$ and $B$ of $f$  are orthogonal to the axis $A_{g}$ of $g$.

{\rm (2)} The endpoints of $A$ are different from the endpoints of $B$.

{\rm (3)} The element $g$ is  parabolic if and only if $A$ and $B$ share the unique endpoint.

When $p\ge 3$,

{\rm (4)} the element $g$ is elliptic if and only if  $A\cap B\neq\emptyset$;

{\rm (5)} the element $g$ is loxodromic if and only if  $A\cap B=\emptyset$.

When $p=2$,

{\rm (6)} the element $g$ is elliptic if and only if the two tailed geodesic lines $A_{x}\cap B_{y}\neq \emptyset$;

{\rm (7)} the element $g$ is loxodromic if and only if the two tailed axes $A_{x}\cap B_{y}=\emptyset$.

\end{lem}

\begin{proof}

 If $g$ is loxodromic or elliptic, without loss of generality, let $g=\lambda^{2}z$. The axis $A_{g}$ is the geodesic line connecting $0$ and $\infty$. Let  $h(z)=-\frac{b^{2}}{z}$ and  $f(z)=g\circ h^{-1}(z)=-\frac{\lambda^{2}b^{2}}{z}$. Thus $h$ has two fixed points $bk,-bk$ and $f$ has two fixed points $\lambda bk, -\lambda bk$, where $k^{2}=-1$. It is easy to see that $f$ and $h$ interchange $0$ and $\infty$, and $f\circ f=h\circ h=z$, namely $f$ and $h$ are two involutions. Hence $A,B$ are orthogonal to $A_{g}$ respectively. We prove {\rm (1), (2)}.

In case of $p\ge 3$, if $A\cap B\neq \emptyset$, since the $\mathbb{P}_{Ber}^{1}$ is a $\mathbb{R}-$tree, there exists a point $x\in A\cap B$ corresponding the disc $D$, which contains  endpoints $bk, -bk, \lambda bk, -\lambda bk$ and $|2bk|=|2\lambda bk|$ which implies that $|\lambda|=1$, i.e. $g$ is elliptic element.  If $A\cap B=\emptyset$, either $bk,-bk\notin D(\lambda bk,|\lambda bk|)$ or $\lambda bk,-\lambda bk\notin D(bk,|bk|)$ which implies that $|\lambda bk-bk|>|bk|$ or $|\lambda bk-bk|>|\lambda bk|$, i.e. $|\lambda-1|>1$ or $|1-\frac{1}{\lambda}|>1$. Hence $g$ is loxodromic. We prove {\rm (4),(5)}.

In case of $p=2$, if the tailed axes $A_{x}\cap B_{y}=\emptyset$, we denote the tails of the tailed axes by $D(bk,|bk|)$ and $D(\lambda bk, |\lambda bk|)$, the two tailed axes $A_{x}$ and $B_{y}$ do not intersect, since $D(bk, |bk|)^{-}\cap D(\lambda bk, |\lambda bk|)^{-}=\emptyset$, and $D(\lambda bk, |\lambda bk|)\neq D(\lambda bk, |\lambda bk|)$. Conversely, if two tailed axes do not intersect, then $|\lambda bk|\neq |bk|$, namely $|\lambda|\neq 1$. This implies that $g$ is a loxodromic element. We prove {\rm (6),(7)}

If $g$ is parabolic,  without loss of generality, we can assume that $g(z)=z+1$. Let $f(z)=-z$ and $h(z)=-z-1$. Then $g(z)=f\circ h(z)=-(-z-1)=z+1$. Since the axis $A$ of $f$ is the geodesic line connecting $\{0,\infty\}$ and the axis $B$ of $h$ is the geodesic line connecting $\{-\frac{1}{2},\infty\}$,  we have $\infty\in A\cap B$.

Conversely, if $A$ and $B$ share only one endpoint, without loss of generality, we can assume that $A$ is the geodesic line connecting $\{0,\infty\}$, and $B$ is the geodesic line connecting $\{0,1\}$. Then $f(z)=-z$ and $h(z)=\frac{z}{2z-1}$ which implies that $g(z)=f\circ h(z)=-\frac{z}{2z-1}$ is a parabolic element, since $g$ has a unique fixed point in $\mathbb{P}^{1}(\mathbb{C}_{p})$. We know that if $A$ and $B$ have different endpoints, then $g$ is either loxodromic or elliptic.

\end{proof}

By the proof of Lemma \ref{thm:1.1}, we know that for any $p$-adic M\"obius map $g$, there exist two half turns $f$ and $h$ such that $g=fh$. Furthermore, following the proof, since $f$ and $h$ are not unique, we can make the axis of $f$ contain the Gauss point $\zeta_{Gauss}$.

We denote the set of fixed points of an element $g$ by $F_{g}=\{x\in \mathbb{P}_{Ber}^{1}|g(x)=x\}$.  Let $a,b\in \mathbb{P}^{1}(\mathbb{C}_{p})$, and $x\in \mathbb{H}_{Ber}$. %We call a tripod $A_{x}$ the {\bf tailed geodesic line} if $A_{x}=l_{1}\cup l_{2}$, where $l_{1}$ is the ray connecting $a,x$,  $l_{2}$ is the ray connecting $b,x$, and $\inf_{y\in A}{\rho(x,y)}=\log2$. We call the point $x$ of $A_{x}$ the tailed point of the tailed geodesic line, and the geodesic line $A$ is the associated geodesic line of the tailed geodesic line.

\begin{pro}Let $g\in \mathrm{PSL}(2,\mathbb{C}_{p})$.

$ (1)$ If  $g$ is a loxodromic element, then the set of fixed points of $g$ contains two points in $\mathbb{P}^{1}(\mathbb{C}_{p})$.

$(2)$  If $g$ is a tame elliptic element, then the set of fixed points of $g$ is a geodesic line in $\mathbb{P}^{1}_{Bek}$, and $F_{g}\cap \mathbb{P}^{1}(\mathbb{C}_{p})$ contains two points.

$(3)$  Let $g$ be a wild elliptic element. Then the interior of the set of the fixed points of $g$ contains a geodesic line in $\mathbb{P}_{Ber}^{1}$.

$(4)$  If $g$ is a parabolic element, then the fixed points of $g$ is an open disc with its boundary with respect to the Berkovich topology.

\end{pro}

\begin{proof} If $g$ is loxodromic or elliptic, we can assume that $g(z)=\lambda^{2}z$ with fixed points $0,\infty$.  If $|\lambda|>1$ or $|\lambda |<1$, then by Lemma \ref{yl:main}, we know that $g$ can not fix any point in $\mathbb{H}_{Ber}$. This proves {\rm (1)}. If $|\lambda|=1$, then $g\in \mathrm{PSL}(2,\mathcal{O}_{p})$ which implies that $g$ fixes every point on the geodesic line connecting $0,\infty$.  Furthermore, in case of $|\lambda-1|=|\lambda+1|=1$, for any $x\in \mathbb{H}_{Ber}$ corresponding to the disc $D(a,r)$, we have $g(D(a,r))=D(\lambda^{2} a, |\lambda^{2}|r)=D(\lambda^{2} a, |\lambda^{2}|r)$, if $gx=x$. Then $|\lambda^{2}-1||a|\le r$ which implies that $|a|\le r$, i.e. $0\in D(a,r)$. Hence $x$ is on the geodesic line connecting $0$ and $\infty$.
This proves {\rm (2)}.  When $|\lambda^{2}-1|<1$, for any $x\in \mathbb{H}_{Ber}$ corresponding to the disc $D(a,r)$, we have $g(D(a,r))=D(\lambda^{2} a, |\lambda^{2}|r)=D(\lambda^{2} a, |\lambda^{2}|r)$, if $gx=x$.  This implies that $|a|\le \frac{r}{|\lambda^{2}-1|}$. Hence $g$ fixes any point in the hyperbolic disc $\mathcal{B}(x,\frac{1}{|\lambda^{2}-1|})$, i.e. the interior of the set of the fixed points of $g$ contains a geodesic line in $\mathbb{P}_{Ber}^{1}$. This proves {\rm (3)}.

If $g$ is parabolic,  we can assume that $g=z+1$.  For any $x\in \mathbb{H}_{Ber}$ corresponding to the disc $D(a,r)$, if $g(x)=x$, then $D(a,r)=D(a+1,r)$ which implies that $r\ge 1$.¡¡This proves {\rm (4)}.

\end{proof}

%If $f$ is a involution, then we can assume that $f(z)=\frac{1}{z}$. By Lemma \ref{yl:3.1}, if $\zeta_{a,r}$ with $|a|>r$, then $f(\zeta_{a,r})=\zeta_{\frac{1}{a},\frac{r}{|a|^{2}}}$. If $\zeta_{a,r}=\zeta_{\frac{1}{a},\frac{r}{|a|^{2}}}$, then $r=\frac{r}{|a|^{2}}$ which implies that $|a|=1$ and $|a-\frac{1}{a}|\le r$. This implies that $a^{2}$ is primitive root.

\begin{lem} If a geodesic line $A$ is orthogonal to the other geodesic line $B$, then $B$ is also orthogonal to $A$.

\end{lem}

\begin{proof} Without loss of generality, we can assume that $A$ is a geodesic line with endpoints $0$ and $\infty$ and $B$ is a geodesic line with endpoints $-1$ and $1$. Then $f(A)=B,f(B)=A$, if $f(z)=\frac{z+1}{z-1}$.
\end{proof}

\begin{lem}\label{yl:ounique} Let a geodesic line $A$ be orthogonal to the other geodesic line $B$.

$(1)$ If $p\ge 3$, then $A$ intersects $B$ at one unique point.

$(2)$ If $p=2$, then $A\cap B=\emptyset$.

\end{lem}

\begin{proof}Without loss of generality, we may assume that $A$ is the geodesic connecting $0$ and $\infty$ and  $B$ is a geodesic line with endpoints  $-\alpha,\alpha$. Let the point $x$ correspond to the disc $D(\alpha, |2\alpha|)$ which lies on the geodesic line $B$ and contains the points $-\alpha,\alpha$.

If $p\ge 3$, then $|2\alpha|=|\alpha|$ which implies that $D(\alpha, |2\alpha|)$ contains $0$. Hence $B$ intersects $A$. Conversely, if $B$ intersects $A$, then there exists a  point $x$ corresponding to a disc $D(0,r)$ containing $\alpha$ or $-\alpha$ which implies that $|\alpha|\le r$. Hence $D(0,r)$ contains both $\alpha$ and $-\alpha$. If there exist two points $x_{1},x_{2}\in A\cap B$ corresponding to two discs $D(0,r_{1})$ and $D(0,r_{2})$ respectively, then either $D(0,r_{1})\subset D(0,r_{2})$ or $D(0,r_{2})\subset D(0,r_{1})$. Without loss of generality, we can assume that $D(0,r_{1})\subset D(0,r_{2})$. Let $l_{1}$ be the geodesic segment connecting $\alpha$ and $x_{2}$, and $l_{2}$ be the geodesic segment connecting $-\alpha$ and $x_{2}$. Hence $ l_{1}\cup l_{2}\subset B$, but $l_{1}\cap l_{2}$ contains a segment containing $x_{1}$ and $x_{2}$. This is a contradiction. Hence $A$ intersects $B$ at a uniquely point.

If $p=2$, then for any $x$ lying on the geodesic line connecting $\alpha,-\alpha$, we have that the disc corresponding to $x$ must contain $\alpha$ or $-\alpha$. Without loss of generality, let $x$ correspond to the disc $D(\alpha, r)$. If $x$ lies on $A$, then $D(\alpha,r)$ contains $0$ which implies that $|\alpha|\le r$. Hence $-\alpha \in D(0,r)$. Since $\zeta$ corresponds to the disc $D(\alpha,|2\alpha|)$ containing both $\alpha$ and $-\alpha$,  the geodesic line contains the segment which connecting $\zeta$ and $x$. This is a contradiction. Hence $A\cap B=\emptyset$.

\end{proof}

We say that $g$ keeps a set $A$ invariant if $g(A)=g^{-1}(A)=A$.

\begin{lem} Let $A_{g}$ be the axis of $g$.  If $g$ is a loxodromic element or an elliptic element, then $g$ keeps the axis $A_{g}$ invariant. Furthermore, $g$ fixes every point of the axis $A_{g}$ if and only if $g$ is an elliptic element.
\end{lem}
\begin{proof} Without loss of generality, let $g=\lambda z$. Hence $A_{g}$ is the geodesic line connecting $0,\infty$. If $|\lambda|>1$, then $g$ maps each disk $D(0,r)$ to $D(0,|\lambda|r)$ which is also on the geodesic line. If $|\lambda|=1$, then $g$ maps each disk $D(0,r)$ to $D(0,|\lambda|r)$ which is the disk $D(0,r)$, namely $g$ fixes the point $\zeta_{0,r}$.

\end{proof}

\begin{lem}\label{yl:unique} If $p\ge 3$, and $g=h\circ f$ is a tame elliptic element, where $h,f$ are two involutions, then two axes of $h$ and $f$ only intersect at a unique point.
\end{lem}
\begin{proof} Without loss of generality, we can assume that $g(z)=\lambda^{2}(z)$. Thus $f(z)=-\frac{b^{2}}{z}$ and $h(z)=-\frac{\lambda^{2}b^{2}}{z}$. By Lemma \ref{yl:1}, we have $|\lambda b-b|=|\lambda-1||b|=|b|$. This implies that two axes of $h$ and $f$ only intersect at the point $\zeta_{0,|b|}$ which corresponds to the disc $D(0,|b|)$.
\end{proof}

\begin{lem}\label{lem:p3wildell}If $p\ge 3$, and $g=h\circ f$ is a wild elliptic element, where $h,f$ are two involutions, then two axes of $h$ and $f$ only intersect on a segment, and this segment belongs to the fixed points of $g$.

\end{lem}

\begin{proof}Without loss of generality, we can assume that $g(z)=\lambda^{2} z$, $f(z)=-\frac{1}{z}$ and $h(z)=-\frac{\lambda}{z}$. Let $A$ be the axis of $f(z)$ and $B$ be the axis of $h(z)$. Hence the endpoints of $A$ are $\{-1,1\}$ and the endpoints of $B$ are $\{-\sqrt{\lambda},\sqrt{\lambda}\}$. Since $p\ge 3$, then $1=|-1-1|=|-\sqrt{\lambda}-\sqrt{\lambda}|$ which implies that $\zeta_{Gauss}$ lies on both the axes $A$ and $B$. Since $|\lambda-1|<1$ and $|\lambda+1|<1$,  we have $\min\{|\sqrt{\lambda}-1|,|\sqrt{\lambda}+1|\}<1$ which implies that there exists a point $x\in A\cap B\cap \mathbb{H}_{Ber}$ which corresponds to the disc $D(1,|\sqrt{\lambda}-1|)$ or the disc $D(1,|\sqrt{\lambda}+1|)$. Hence $A\cap B$ contains either the segment connecting $\zeta_{1,|\sqrt{\lambda}-1|}$ and $\zeta_{0,1}$ or the segment connecting $\zeta_{1,|\sqrt{\lambda}+1|}$ and $\zeta_{0,1}$. This segment belongs to the fixed points of fixed points of $f$.

\end{proof}

\begin{lem}\label{yl:orthogonal} If $p\ge 3$, and  $A$ and $B$ are two geodesic lines with four distinct endpoints, then there exists a unique geodesic line which is orthogonal to $A$ and $B$ simultaneously.
\end{lem}
\begin{proof} Without loss of generality, we can assume that $A$ is the geodesic with endpoints $0$ and $\infty$, and $B$ is the other geodesic line with endpoints $a$ and $b$. If $A$ does not intersect $B$, then we have $|a-b|<\max\{|a|,|b|\}$. By the ultrametric property, we have $|a|=|b|>|a-b|$. Let $C$ be a geodesic line with endpoints $\zeta,-\zeta$. By Theorem \ref{thm:1.1}, we have that the geodesic line $C$ is orthogonal to the geodesic line $A$. Let $g=\frac{z-a}{z-b}$. Then the geodesic line $B$ is mapped to the geodesic line $A$ by $g$. If the geodesic $g(C)$ is also orthogonal to $A$, then $g(-\zeta)+g(\zeta)=0$ and $g(C)\cap A\neq \emptyset$. This implies that $\frac{\zeta-a}{\zeta-b}+\frac{-\zeta-a}{-\zeta-b}=0$, namely $\zeta=\sqrt{ab}$. Then the geodesic line $C$ connecting $-\sqrt{ab},\sqrt{ab}$ is orthogonal to both $A$ and $B$ simultaneously.

\end{proof}

\begin{lem}\label{yl:orthogonalp2} If $p= 2$, and  $A_{x}$ and $B_{y}$ are two tailed geodesic lines with four distinct endpoints, then there exists a unique geodesic line which is orthogonal to $A$ and $B$ simultaneously.
\end{lem}

\begin{proof} Without loss of generality, we can assume that the endpoints of $A$ are $-1,1$, and the endpoints of $B$ are $t,s$. We claim that we can find a $p$-adic M\"obius map $f=\frac{az+b}{cz+d}$ such that $f(-1)=-1$, $f(-1)=-1$, $f(t)+f(s)=0$.

Since $f(-1)=-1$, $f(-1)=-1$, we have $a=d,b=c$, and $\frac{at+b}{ct+d}+\frac{as+b}{cs+d}=0$ which yields that $2abst+(a^{2}+b^{2})(s+t)+2ba=0$. We can lift the solution to the projective space, namely $2ABst+(A^{2}+B^{2})(s+t)+AB=0$, and $A^{2}-B^{2}=C^{2}$. Since any two curves in the projective space  intersect, we have solutions in the projective space. If the solution is $(A:B:0)$, namely $C=0$, then $A=B$ or $A=-B$. This implies that $st+s+t+1=0$ or $st-(s+t)+1=0$, namely $s=-1$ or $t=-1$ or $s=1$ or $t=1$. This contradicts that two tailed geodesic line have no common endpoints. Hence $C\neq 0$, namely there exists $p$-adic M\"obius map $f$ such that $f(-1)=-1$, $f(-1)=-1$, $f(t)+f(s)=0$.

Hence we can assume that the tailed geodesic line $A$ has the endpoints $-1,1$ and the tailed geodesic line has the endpoints $-\lambda,\lambda$. Then the two tailed geodesic line are orthogonal to the line connecting $0,\infty$ simultaneously.

\end{proof}

\begin{lem}\label{yl:tail} If $p=2$, and $A_{x}$ is a tailed geodesic line with the tail $x\in B$, and $A\cap B=
\emptyset$, then there exists a tailed geodesic line $B_{y}$ such that $B_{y}$ is orthogonal to $A$, $y\in A$ and $B\subset B_{y}$.

\end{lem}

\begin{proof} Let $y$ be the point on the geodesic line $A$ satisfying $\rho(x,y)=\log2$, and $l$ be the segment connecting $x$ and $y$ such that $B_{y}=l\cup B$ is the tailed geodesic line satisfying the condition.

\end{proof}

\begin{lem}\label{yl:unique2} If $p=2$, and $g=h\circ f$ is a tame elliptic element, where $h,f$ are two involutions, then two tailed geodesic lines of $h$ and $f$ only intersect at a unique point.
\end{lem}
\begin{proof} By Lemma \ref{yl:orthogonalp2}, we can assume that the fixed points of $f$ are $-1,1$ and the fixed points of $h$ are $-\lambda,\lambda$. Since $g$ is a tame elliptic element, we have $|\lambda-1|=1$. Hence the tailed point of the tailed geodesic line of $h$ is $\zeta_{0,1}$, and the tailed point of the tailed geodesic line of $f$ is also $\zeta_{0,1}$. Since $D(1,1)^{-}\cap D(\lambda,1)^{-}=\emptyset$, then the two tailed geodesic lines intersect the unique point $\zeta_{0,1}$.

\end{proof}

We give the following lemma without proof, which follows from Lemma\ref{lem:p3wildell} and Lemma \ref{yl:orthogonalp2} directly.

\begin{lem}\label{lem:p2wildell}If $p=2$, and $g=h\circ f$ is a wild elliptic element, where $h,f$ are two involutions, then two tailed axes $A_{x}, B_{y}$  of $h$ and $f$ only intersect on a segment, and this segment belongs to the fixed points of $g$.

\end{lem}

Proof of Theorem \ref{thm:reduction}

\begin{proof}

If $\bigcap_{g\in G}F_{g}=\emptyset$, then there exist finitely many elements $g_{1},\ldots,g_{n}$ such that $\bigcap_{i=1}^{n}F_{g_{i}}=\emptyset$, since the Berkovich space is compact with respect to the weak topology. Hence if we can show that $\bigcap_{i=1}^{n}F_{g_{i}}\neq \emptyset$ for any positive integer $n$, then we prove the theorem.

Let $f,g$ be two elliptic elements, and denote the axes of $f,g$ by $A_{f}, A_{g}$ respectively.
When $p\ge 3$, by Lemma \ref{yl:orthogonal}, we have that there exists a involution $a$ whose axis $A$ is orthogonal to $A_{f}$ and $A_{g}$ simultaneously. By Lemma \ref{thm:1.1}, there exist two involutions $b,c$ such that $f=a\circ b$ and $g=a\circ c$. We denote the set of the fixed points of $f,g$ by $F_{f}, F_{g}$ respectively, and the axes of $a,b,c$ by $A,B,C$ respectively. By Lemma \ref{lem:p3wildell}, we know that $F_{f}\supset A\cap B\neq \emptyset$, and $F_{g}\supset A\cap C\neq \emptyset$, and $B\cap C\neq \emptyset$, since $h=f^{-1}\circ g$ is elliptic. Choosing $x\in  A\cap B, y\in A\cap C, z\in B\cap C$, there exists $w\in [x,z]\cap[y,z]\cap[x,y]$, since $\mathbb{P}^{1}_{Ber}$ is an $\mathbb{R}-tree$. This means that $F_{f}\cap F_{g}\neq \emptyset$.

By induction, $\bigcap_{i=1,i\neq k}^{n}F_{g_{i}}\neq \emptyset$ for $k=1,\ldots,n$, and then we want to show $\bigcap_{i=1}^{n}F_{g_{i}}\neq \emptyset$. Since $\bigcap_{i=1,i\neq n-1}^{n}F_{g_{i}}\neq \emptyset, \bigcap_{i=1,i\neq n}^{n}F_{g_{i}}\neq \emptyset, F_{g_{n-1}}\cap F_{g_{n}}\neq \emptyset$, choosing $x\in \bigcap_{i=1,i\neq n-1}^{n}F_{g_{i}}, y\in \bigcap_{i=1,i\neq n}^{n}F_{g_{i}}, z\in F_{g_{n-1}}\cap F_{g_{n}}$, there exists $w\in [x,z]\cap[y,z]\cap[x,y]$ such that $w\in \bigcap_{i=1}^{n}F_{g_{i}}$. This implies that each element in $G$ share at least one fixed point.

When $p=2$, by Lemma \ref{yl:orthogonal}, we have that there exists a involution $a$ whose axis $A$ is orthogonal to $A_{f}$ and $A_{g}$ simultaneously. By Lemma \ref{thm:1.1}, there exist two involutions $b,c$ such that $f=a\circ b$ and $g=a\circ c$. We denote the set of the fixed points of $f,g$ by $F_{f},F_{g}$. Thanks to Lemma \ref{yl:orthogonalp2} and Lemma \ref{yl:tail}, there exist two tailed geodesic line $A_{x}$, $A_{y}$ who share the same associated geodesic line $A$ which are orthogonal to two axes $A_{f}$ and $A_{g}$. We denote the tailed axes of $b,c$ by $B_{x}, C_{y}$.  By Lemma \ref{lem:p2wildell}, we know that $F_{f}\supset A_{x}\cap B_{x}\neq \emptyset$, and $F_{g}\supset A_{y}\cap C_{y}\neq \emptyset$, and $B_{x}\cap C_{y}\neq \emptyset$, since $h=f^{-1}\circ g$ is elliptic. Choosing $u\in  A_{x}\cap B_{x}, v\in A_{y}\cap C_{y}, w\in B_{y}\cap C_{y}$, there exists $\omega\in [u,v]\cap[u,w]\cap[v,w]$, since $\mathbb{P}^{1}_{Ber}$ is an $\mathbb{R}-tree$. This means that $F_{f}\cap F_{g}\neq \emptyset$.

Following the proof of the case $p=2$, it is obviously that when $p=2$, each element in $G$ share at least one fixed point.

By conjugation, we can assume that each element in $G$ shares the unique fixed point $\zeta_{Gauss}$. By Lemma \ref{yl:gdguass}, we know that each element in $G$ has good reduction which yields that $G$ has a potentially good reduction. By Lemma \ref{yl:gdrdpM}, we know that if $g$ has good reduction, then $g\in \mathrm{PSL}(2,\mathcal{O})$. Since each element $f\in G$ can be written as $\phi f'\phi^{-1}$, where $f'$ has good reduction and $\phi\in \mathrm{PSL}(2,\mathbb{C}_{p})$,  $\rho_{v}(f(x),f(y))=\rho_{v}(\phi f'\phi^{-1}(x),\phi f'\phi^{-1}(y))\le L_{1}\rho_{v}( f'\phi^{-1}(x), f'\phi^{-1}(y))\le L_{1}\rho_{v}( \phi^{-1}(x), \phi^{-1}(y))\le L_{1}L_{2}\rho_{v}(x,y)$, where $L_{1},L_{2}$ depending only on $\phi$. Hence $G$ is equicontinuous on $\mathbb{P}^{1}(\mathbb{C}_{p})$.

\end{proof}

%\begin{cor} If $G$ is a discrete subgroup of $\mathrm{PSL}(2,\mathbb{Q}_{p})$ with $p>3$ and $G$ contains elliptic elements only, then $G$ has potentially good reduction.

%\end{cor}

%\begin{proof} We can lift each element $g(z)=\frac{az+b}{cz+d}$ to an element $\bar{g}$ in $\mathrm{SL}(2,\mathbb{Q}_{p})$. We denote \begin{displaymath}
%\bar{g}=\left(\begin{array}{cc}
%a&b    \\
 %c&d
%\end{array}\right).
%\end{displaymath} We know that the eigenvalue $\lambda$ satisfying the equation $\lambda+1/\lambda=a+d$.
%If $g$ is a wild elliptic element, then $\lambda^{p^{d}}=1$ with $d\ge 1$.  Since $p>3$, we have $[\mathbb{Q}_{p}(\lambda):\mathbb{Q}_{p}]=p^{d-1}(p-1)> 2$. Therefore there does not exist such kind of $\lambda$. Namely, $G$ contains no wild elliptic elements. Then $G$ has potentially good reduction.

%\end{proof}

\begin{thm} If $G$ is a discrete subgroup of $\mathrm{PSL}(2,\mathbb{C}_{p})$ and the limit set of $G$ is empty, then $G$ has potentially good reduction.

\end{thm}

\begin{proof} We have that $G$ contains no loxodromic element $g$, since the fixed points of $g$ are in the limit set of $G$ which yields that $G$ contains parabolic elements and elliptic elements only. Since $G$ is a discrete subgroup of $\mathrm{PSL}(2,\mathbb{C}_{p})$,  $G$ contains no parabolic elements. By Theorem \ref{thm:reduction}, we know that $G$ has potentially good reduction.

\end{proof}

%\begin{thm} If $G$ has potentially good reduction, then $G$ contains elliptic elements and parabolic elements only.

%\end{thm}

%\begin{proof} By conjugation, $G$ has good reduction. By Lemma \ref{yl:gdrdpM}, we have that each element in $\mathrm{PSL}(2,\mathcal{O})$. This implies that each element in $G$ fixes the point $\mathbb{H}_{Ber}$. We know that each loxodromic element does not fix any point in $\mathbb{H}_{Ber}$. Hence each element in $G$ is either an elliptic element or a parabolic element.

%\end{proof}

\begin{ex}Let $f_{n}(z)=z+p^{-n}\;\;(n\ge 1)$, and the group $G$ is generated by each $f_{n}$. Then $G$ contains parabolic elements only and does not have potentially good reduction.

\end{ex}

\begin{proof} For any disc $D(a,r)$ which is fixed by $f_{n}$, we have $r\ge p^{n}$. Since $n$ is arbitrary,  the only point fixed by $G$ is the $\infty$. Since each generator can commutate with each other, we know that each element in $G$ can only fixed the unique point $\infty$ in $\mathbb{P}_{Ber}^{1}$.

\end{proof}

\begin{ex} Let $G\subset \mathrm{PSL}(2,\mathbb{C}_{p})$, and $\zeta_{i}$ be the  $p^{i}$-th primitive root of unity. Suppose that $G$ is generated by
 \begin{displaymath}
g_{i}=\left(\begin{array}{cc}
\zeta_{i}&0\\
0&\zeta_{i}^{-1}
\end{array}\right),
\end{displaymath} for all the positive integer $i\ge 1$. Then $G$ is discrete, and the limit set $\Lambda(G)$ of $G$ is $\{0,\infty\}$ is a compact set.
\end{ex}

\begin{proof} In \cite{QY}, we have proved that $G$ is a discrete subgroup of $\mathrm{PSL}(2,\mathbb{C}_{p})$. Furthermore, the points $\{0,\infty\}$ are the fixed points of all the elements $g_{i}$, $i\ge 1$, namely $0,\infty$ are the limit sets of $G$

\end{proof}

\section{Norms of $p$-adic M\"obius maps and its applications}

%For any two $p$-adic M\"obius maps $g,h$, we denote the {\bf uniformly convergent  metric} on the $\mathrm{PSL}(2,\mathbb{P}^{1}(\mathbb{C}_{p}))$ by $\rho_{0}(g,h)=\displaystyle\mathop{\sup}_{z\in \mathbb{P}^{1}(\mathbb{C}_{p})}\rho_{v}(g(z),h(z))$.

\begin{pro}\label{pro:rv} Suppose that  $f,g,h\in \mathrm{PSL}(2,\mathbb{C}_{p})$. Then

{\rm (1)}$\rho_{0}(fh,gh)=\rho_{0}(f,g)$, and $\rho_{0}(hf,hg)\le L(h)\rho_{0}(f,g)$.

{\rm (2)}If $h\in \mathrm{PSL}(2,\mathcal{O}_{p})$,  $\rho_{0}(h^{-1}fh,h^{-1}gh)=\rho_{0}(f,g)$.

\end{pro}

\begin{proof} Since $h$ is an automorphism on $\mathbb{P}^{1}(\mathbb{C}_{p})$, we have $$\rho_{0}(fh,gh)
=\displaystyle\mathop{\sup}_{z\in \mathbb{P}^{1}(\mathbb{C}_{p})}\rho_{v}(fh(z),gh(z))=\displaystyle\mathop{\sup}_{w=h(z)\in \mathbb{P}^{1}(\mathbb{C}_{p})}\rho_{v}(f(w),g(w))=\rho_{0}(f,g).$$ Since $\rho_{v}(h(z),h(w))\le L(h)\rho_{v}(z,w)$, we have $\rho_{v}(hf(z),hg(z))\le L(h)\rho_{v}(f(z),g(z))$ which yields that $\rho_{0}(hf,hg)\le L(h)\rho_{0}(f,g)$.

Since $$\rho_{v}(h^{-1}fh,h^{-1}gh)=\rho_{v}(h^{-1}f,h^{-1}g)\le L(h) \rho_{v}(f,g)\le L(h)L(h^{-1})\rho_{v}(h^{-1}f,h^{-1}g),$$ we have $\rho_{v}(h^{-1}f,h^{-1}g)\le L(h) \rho_{v}(f,g)\le \rho_{v}(h^{-1}f,h^{-1}g)$ which yields that $\rho_{0}(h^{-1}fh,h^{-1}gh)=\rho_{0}(f,g)$.

\end{proof}

Let $m(g)=\parallel g-g^{-1}\parallel$ and $M(g)=\frac{\parallel g-g^{-1}\parallel}{\parallel g\parallel}$.

\begin{pro}\label{pro:monotune}Let $p$ be a prime number. Then $p^{-\frac{1}{p-1}}\ge 2^{-1}$.

\end{pro}

\begin{proof}Let $f(x)=x^{-\frac{1}{x-1}}$. Then $f'(x)=f(x)(\ln x-(1-\frac{1}{x}))/(x-1)^{2}$ which yields that $f'(x)>0$ if $x\ge 3$. Since $f(3)=3^{-\frac{1}{2}}\ge 2^{-1}=f(2)$, we have $p^{-\frac{1}{p-1}}\ge 2^{-1}$ if $p$ is a prime number.

\end{proof}

\begin{thm}\label{thm:ucM}Let $p\ge3$, and $g$ be a $p$-adic M\"obius map. Then $\rho_{0}(g,I)=M(g)$.

\end{thm}

\begin{proof} By Lemma \ref{lem:lip1},  there exists an element $h\in \mathrm{PSL}(2,\mathcal{O}_{p})$ such that $hgh^{-1}=\frac{az+b}{d}$ with $ad=1$. By Proposition \ref{pro:rv}, we have $\rho_{0}(hgh^{-1},I)=\rho_{0}(g,I)$. By Theorem \ref{thm:1}, we know that $\parallel g-g^{-1}\parallel= \parallel h(g-g^{-1})h^{-1}\parallel=\parallel hgh^{-1}-hg^{-1}h^{-1}\parallel$. Thus we can rewrite  $hgh^{-1}$ as $g$. Hence $$M(g)=\frac{\max\{|a-d|,|b|\}}{\max\{|a|,|d|,|b|\}}.$$
Moreover $$\rho_{v}(g(z),z)=\frac{|a^{2}z+ab-z|}{\max\{1,|z|\}\max\{1,|a^{2}z+ab|\}}.$$

If $g$ is parabolic, then $$\rho_{v}(g(z),z)=\frac{|b|}{\max\{1,|z|\}\max\{1,|z+b|\}}\le M(g),$$ and $$\rho_{0}(g,I)\ge \rho_{v}(g(0),0)=M(g).$$ Thus $\rho_{0}(g,I)=M(g)$.

If $g$ is loxodromic, we can assume that $|a|>1$. Since $$\begin{aligned}&\rho_{v}(g(0),0)=\frac{|ab|}{\max\{1,|ab|\}},\rho_{v}(g(1),1)=\frac{|a^{2}+ab-1|}{\max\{1,|a^{2}+ab|\}}=\frac{|a^{2}+ab|}{\max\{1,|a^{2}+ab|\}},\\ &\max\{|ab|,|a^{2}+ab|\}\ge 1,
\end{aligned}$$ we have $M(g)=1\ge \rho_{0}(g,I)\ge \max\{\rho_{v}(g(0),0),\rho_{v}(g(1),1)\}=1=M(g).$

If $g$ is elliptic, we have $|a^{2}-1|\le1$. If $|ab|> |a^{2}-1|$, then $M(g)=|ab|\ge\rho_{0}(g,I)\ge \rho_{v}(g(0),0)=|ab|=M(g)$.
If $|ab|\le |a^{2}-1|$, there exists a number $\omega\in \mathcal{O}_{p}$ such that $|(a^{2}-1)\omega+ab|=|a^{2}-1|$ which yields that $$M(g)=|a^{2}-1|\ge\rho_{0}(g,I)\ge \rho_{v}(g(\omega),\omega)=|a^{2}-1|=M(g).$$

\end{proof}

If $p=2$, then $2^{-1}M(g)\le \frac{\max\{|a-d|,|b|\}}{\max\{|a|,|d|,|b|\}} \le 2M(g)$. Thus we give the following theorem without proof.

\begin{thm}\label{thm:ucMp=2}Let $p=2$, and $g$ be a $p$-adic M\"obius map. Then $2^{-1}M(g)\le \rho_{0}(g,I)\le 2M(g)$.

\end{thm}

\begin{thm}\label{thm:matr} For any $p$-adic M\"obius map $g$,  $\rho_{0}(g,I)\le \parallel g-I\parallel$.

\end{thm}

\begin{proof} Following Theorem \ref{thm:ucM}, we can assume that $g(z)=\frac{az+b}{d}$ with $ad=1$. Thus $\|g-I\|=\max\{|a-1|,|b|,|d-1|\}$. Moreover $$\rho_{v}(g(z),z)=\frac{|a^{2}z+ab-z|}{\max\{1,|z|\}\max\{1,|az+b|\}}\le 1.$$

If $g$ is loxodromic, we can assume $|a|>1$. Hence $\max\{|a-1|,|d-1|\}>1$ which yields that $\rho_{0}(g,I)\le \parallel g-I\parallel$.

If $g$ is parabolic, then $a=1$. We have $$\rho_{v}(g(z),z)=\frac{|b|}{\max\{1,|z|\}\max\{1,|z+b|\}}\le |b|=\parallel g-I\parallel.$$

If $g$ is elliptic, then $|a|=1$.  We have $$\begin{aligned}&\rho_{v}(g(z),z)=\frac{|a^{2}z+ab-z|}{\max\{1,|z|\}\max\{1,|az+b|\}}\le \frac{|a^{2}z+ab-z|}{\max\{1,|z|\}}\\
&\le \frac{\max\{|(a^{2}-1)z|,|ab|\}}{\max\{1,|z|\}}\le \max\{|a^{2}-1|,|ab|\} \le \parallel g-I\parallel.\end{aligned}$$

\end{proof}

Let $\varepsilon(g)=\max\{\rho_{v}(g(z_{0}),z_{0}),\rho_{v}(g(z_{1}),z_{1}),\rho_{v}(g(z_{2}),z_{2})\}$, where $z_{0},z_{1},z_{2}$ are three distinct roots of the equation $z^{3}=1$.

\begin{thm}\label{thm:ine1}For any $p$-adic M\"obius map $g$, we have $2^{-1}\varepsilon(g)\le M(g)\le 6\varepsilon(g)$.

\end{thm}

\begin{proof} Since $\varepsilon(g)=\max\{\rho_{v}(g(z_{0}),z_{0}),\rho_{v}(g(z_{1}),z_{1}),\rho_{v}(g(z_{2}),z_{2})\}\le \rho_{0}(g,I)$, by Theorem \ref{thm:ucM} and \ref{thm:ucMp=2}, we have $\varepsilon(g)\le \rho_{0}(g,I)\le 2M(g)$.

Since $z_{0},z_{1},z_{2}$ are three distinct roots of the equation $z^{3}=1$, and let $z_{0}=1$, we have $z_{1}+z_{2}+1=0$, and by Lemma\ref{yl:1}, Lemma \ref{yl:2} and Proposition \ref{pro:monotune}, we have $2^{-1}\le |z_{1}-z_{2}|\le 1$. This implies that
$$\begin{aligned}&|3b|=|cz_{1}+(d-a)z_{2}-b+cz_{2}+(d-a)z_{1}-b+cz_{0}+(d-a)z_{0}-b|\\
&\le \max\{|cz_{1}+(d-a)z_{2}-b|, |cz_{2}+(d-a)z_{1}-b|, |cz_{0}+(d-a)z_{0}-b|\}.\end{aligned}$$ We denote $\varepsilon'(g)$ by $\max\{|cz_{1}+(d-a)z_{2}-b|, |cz_{2}+(d-a)z_{1}-b|, |cz_{0}+(d-a)z_{0}-b|\}$. Hence $|b|\le 3\varepsilon'(g)$ which yields that $$\max\{|cz_{1}+(d-a)z_{2}|, |cz_{2}+(d-a)z_{1}|, |cz_{0}+(d-a)z_{0}|\}\le 3\varepsilon'(g).$$ Thus $$\max\{|cz_{1}+(d-a)z_{2}+cz_{2}+(d-a)z_{1}|,|cz_{1}+(d-a)z_{2}-cz_{2}-(d-a)z_{1}|\}\le 3\varepsilon'(g),$$ namely $|c+(d-a)|\le 3\varepsilon'(g)$ and $|c(\frac{z_{1}}{z_{2}})+(d-a)|\le 3\varepsilon'(g)$ which yields that $|c(\frac{z_{1}}{z_{2}}-1)|\le 3\varepsilon'(g)$. This implies that $|c|\le 6\varepsilon'(g)$ and $|(d-a)|\le 6\varepsilon'(g)$, namely $\max\{|a-d|,|b|,|c|\}\le 6\varepsilon'(g)$. For any $z$ with $|z|=1$, we have $\max\{|az+b|,|cz+d|\}\le \max\{|az|, |b|,|cz|, |d|\}=\parallel g\parallel$ which yields that $\max\{|a-d|,|2b|,|2c|\}\le\max\{|a-d|,|b|,|c|\}\le 6\varepsilon'(g)$. This implies that $\max\{|a-d|,|2b|,|2c|\}/\parallel g\parallel\le\max\{|a-d|,|b|,|c|\}/\parallel g\parallel\le 6\varepsilon'(g)/\max\{|az+b|,|cz+d|\}\max\{1,|z|\}$ for any $z$ with $|z|=1$.
Since $$\begin{aligned}&\frac{\varepsilon'(g)}{\max\{|az+b|,|cz+d|\}\max\{1,|z|\}}\le \max\{\frac{|cz^{2}_{0}+(d-a)z_{0}-b|}{\max\{|az_{0}+b|,|cz_{0}+d|\}\max\{1,|z_{0}|\}},\\
&\frac{|cz^{2}_{1}+(d-a)z_{1}-b|}{\max\{|az_{1}+b|,|cz_{1}+d|\}\max\{1,|z_{1}|\}},\frac{|cz^{2}_{2}+(d-a)z_{2}-b|}{\max\{|az_{2}+b|,|cz_{2}+d|\}\max\{1,|z_{2}|\}}\}\\
&\le \varepsilon(g),\end{aligned}$$ we have $M(g)\le 6 \varepsilon(g)$.
\end{proof}

\begin{cor}\label{cor:ine2}For any $p$-adic M\"obius map $g$, $\frac{1}{4}\varepsilon(g)\le \rho_{0}(g,I)\le \varepsilon(g)$.

\end{cor}

\begin{proof}By Theorem\ref{thm:ucM} and Theorem \ref{thm:ine1}, we get the inequality directly.

\end{proof}

Let $\varepsilon_{1}(g)=\{\rho_{v}(g(0),0),\rho_{v}(g(1),1),\rho_{v}(g(\infty),\infty)\}$.

\begin{thm}For any $p$-adic M\"obius map, $2^{-1}\varepsilon_{1}(g)\le M(g)\le \varepsilon_{1}(g)$.

\end{thm}

\begin{proof}Since $\varepsilon(g)=\max\{\rho_{v}(g(0),0),\rho_{v}(g(1),1),\rho_{v}(g(\infty),\infty)\}\le \rho_{0}(g,I)$, by Theorem \ref{thm:ucM}, we have $\varepsilon_{1}(g)\le \rho_{0}(g,I)\le 2M(g)$.

Since $$\rho_{v}(g(z),z)=\frac{|cz^{2}+(d-a)z-b|}{\max\{|az+b|,|cz+d|\}\max\{1,|z|\}},$$ we have $$\rho_{v}(g(0),0)=\frac{|b|}{\max\{|b|,|d|\}}, \rho_{v}(g(\infty),\infty)=\frac{|c|}{\max\{|a|,|c|\}},$$ and $$\rho_{v}(g(1),1)=\frac{|c+(d-a)-b|}{\max\{|a+b|,|c+d|\}}.$$  Since $\max\{|a-d|,|2b|,|2c|\}\le \max\{|b|,|c|,|c+(d-a)-b|\}$ and $\max\{|a+b|,|c+d|\}\le \max\{|a|,|c|,|b|,|d|\}=\parallel g\parallel,$ we have $$M(g)=\frac{\max\{|a-d|,|2b|,|2c|\}}{\parallel g\parallel}\le \frac{\max\{|b|,|c|,|c+(d-a)-b|\}}{\parallel g\parallel}$$

$$\le \max\{\frac{|b|}{\max\{|b|,|d|\}},\frac{|c|}{\max\{|a|,|c|\}},\frac{|c+(d-a)-b|}{\max\{|a+b|,|c+d|\}}\}=\varepsilon_{1}(g)$$.

\end{proof}

Let $\varepsilon_{2}(g)=\max\{\rho_{v}(g(0),0),\rho_{v}(g(\infty),\infty)\}$.

\begin{cor}If $g$ is a parabolic element, then $2^{-1}\varepsilon_{2}(g)\le M(g)\le \varepsilon_{2}(g)$.

\end{cor}

\begin{proof}Since $\varepsilon_{2}(g)<\varepsilon_{1}(g)$, we have $2^{-1}\varepsilon_{2}(g)\le 2^{-1}\varepsilon_{1}(g)\le M(g)$.

By Proposition \ref{pro:tr}, if $g$ is a parabolic element, then $(a+d)^{2}=4$ which yields that $4bc=-(a-d)^{2}$. Thus it implies that $|a-d|=\sqrt{|4||bc|}\le \sqrt{|bc|}\le \max\{|b|,|c|\}$ which yields that $\max\{|a-d|,|2b|,|2c|\}\le \max\{|b|,|c|\}$. Hence $$M(g)=\frac{\max\{|a-d|,|2b|,|2c|\}}{\parallel g\parallel}\le \frac{\max\{|a-d|,|b|,|c|\}}{\parallel g\parallel}$$ $$\le \max\{\frac{|b|}{\max\{|b|,|d|\}},\frac{|c|}{\max\{|a|,|c|\}}\}=\varepsilon_{2}(g)=\max\{\rho_{v}(g(0),0),\rho_{v}(g(\infty),\infty)\}.$$

\end{proof}

As  applications of these inequalities, we can get the convergence theorem of $p$-adic M\"obius maps.

Let $\{f_{n}\}$ be a sequence of $p$-adic M\"obius maps, and ${\bf U}$ be the set of points at which the sequence $\{f_{n}\}$ converges pointwisely, and $f=\displaystyle\mathop{\lim}_{n\rightarrow\infty}f_{n}$ on ${\bf U}$. Write $f_{n} \hookrightarrow ({\bf U}, f)$ to mean that {\bf U} is the set of convergence of $f_{n}$ and that $f_{n}\rightarrow f$ on (and only on) {\bf U}. In \cite{WY}, we proved the following theorem.

\begin{thm}[\cite{WY}]\label{thm:convergence}Suppose that there exists a sequence $p$-adic M\"obius maps $f_{n}$ such that $f_{n} \hookrightarrow ({\bf U}, f)$ with ${\bf U}\neq \emptyset$. Then one of the following possibilities
occurs:

$(a)$ ${\bf U}=\mathbb{P}^{1}(\mathbb{C}_{p})$, and $f$ is a $p$-adic M\"obius map;

$(b)$ ${\bf U}=\mathbb{P}^{1}(\mathbb{C}_{p})$, and $f$ is constant on the complement of one point on ${\bf U};$

$(c)$ ${\bf U}= \{z_{1}, z_{2}\}$ and $f(z_{1})\neq f(z_{2})$; or

$(d)$ $f$ is constant on {\bf U}.

\end{thm}

We can reprove this theorem by the use of the three-point norms.

If ${\bf U}$ contains only one point, it is the case $(d)$, and if ${\bf U}$ contains two points only, it is the case $(c)$ or $(d)$. Hence we only need to consider the case when ${\bf U}$ contains at least three points.

We prove the following theorem by using different norms of  $p-$adic M\"obius maps without using the cross ratios of $p-$adic M\"obius maps.

\begin{thm}\label{thm:threeptconv}Let $\{f_{n}\}$ be a sequence of $p$-adic M\"obius maps and $z_{j},j=1,2,3$ be three distinct points with $f_{n}(z_{j})\rightarrow w_{j}$, where $w_{j}$ are also three distinct points. Then a sequence $\{f_{n}\}$ converge to a $p$-adic M\"obius map $f$ uniformly, where $f(z_{j})=w_{j},j=1,2,3$.

\end{thm}

\begin{proof} We can find a $p$-adic M\"obius map $h$ such that $h(z_{1})=u_{1},h(z_{2})=u_{2},h(z_{3})=u_{3}$, where $u_{i}$ are the three distinct roots of $z^{3}=1$. Then $hf_{n}h^{-1}(u_{i})\rightarrow w_{i}$.  By Corollary \ref{cor:ine2}, we know that $\frac{1}{4}\varepsilon(hf^{-1}f_{n}h^{-1})\le \rho_{0}(hf^{-1}f_{n}h^{-1},I)\le\varepsilon(hf^{-1}f_{n}h^{-1})$. This yields that $hg^{-1}f_{n}h^{-1}$ converges to $I$ uniformly, namely $f_{n}$ converges to $f$ uniformly.

\end{proof}
\begin{pro}[\cite{WY}]\label{pro:crossratio} Let $f\in \mathrm{PSL}(2,\mathbb{C}_{p})$. Then $f$ preserves the chordal cross ratio, namely $$\frac{\rho_{v}(f(x),f(y))\rho_{v}(f(z),f(w))}{\rho_{v}(f(x),f(z))\rho_{v}(f(y),f(w))}=\frac{\rho_{v}(x,y)\rho_{v}(z,w)}{\rho_{v}(x,z)\rho_{v}(y,w)}.$$

\end{pro}
\begin{thm}\label{thm:twoptconv} Let $\{f_{n}\}$ be a sequence of p-adic maps, and suppose that there exist three distinct points $x_{1}, x_{2}, x_{3}$ in $\mathbb{P}^{1}(\mathbb{C}_{p})$ such that $\displaystyle\mathop{\lim}_{n\rightarrow\infty}f_{n}(x_{1})=\displaystyle\mathop{\lim}_{n\rightarrow\infty}f_{n}(x_{2})=\alpha$,
$\displaystyle\mathop{\lim}_{n\rightarrow\infty}f_{n}(x_{3})=\beta$, where $\alpha\neq \beta$.
Then $f_{n}\rightarrow \alpha$ on $\mathbb{P}^{1}(\mathbb{C}_{p})\setminus{x_{3}}$, namely ${\bf U}=\mathbb{P}^{1}(\mathbb{C}_{p})$, and $f$ is constant on the complement of one point on ${\bf U}$.

\end{thm}

\begin{proof} Without loss of generality, we can assume that $x_{1}=0,x_{2}=1,x_{3}=\infty$ and $\alpha=0,\beta=\infty$. Assuming $x_{4}\in \mathbb{P}^{1}(\mathbb{C}_{p})\setminus \{0,1,\infty\}$, if the sequence $\{f_{n}(x_{4})\}$ does not converge to $0$, there exists a subsequence $\{f_{n_{j}}(x_{4})\}$ and a fixed positive number $\delta$ such that $|f_{n_{j}}(x_{4})|>\delta$. $$\frac{\rho_{v}(f_{n_{j}}(0),f_{n_{j}}(1))\rho_{v}(f_{n_{j}}(\infty),f_{n_{j}}(x_{4}))}{\rho_{v}(f_{n_{j}}(0),f_{n_{j}}(\infty))\rho_{v}(f_{n_{j}}(1),f_{n_{j}}(x_{4}))}=\frac{\rho_{v}(0,1)\rho_{v}(\infty,x_{4})}{\rho_{v}(0,\infty)\rho_{v}(1,x_{4})}\neq 0.$$

Letting $n_{j}$ tend to $\infty$,

$$0=\displaystyle\mathop{\lim}_{n_{j}\rightarrow\infty}\frac{\rho_{v}(f_{n_{j}}(0),f_{n_{j}}(1))\rho_{v}(f_{n_{j}}(\infty),f_{n_{j}}(x_{4}))}{\rho_{v}(f_{n_{j}}(0),f_{n_{j}}(\infty))\rho_{v}(f_{n_{j}}(1),f_{n_{j}}(x_{4}))}=\frac{\rho_{v}(0,1)\rho_{v}(\infty,x_{4})}{\rho_{v}(0,\infty)\rho_{v}(1,x_{4})}\neq 0.$$
This is a contradiction. Hence $\displaystyle\mathop{\lim}_{n\rightarrow\infty}f_{n}(x_{4})=0$.
\end{proof}

Combing Theorem \ref{thm:threeptconv} and Theorem \ref{thm:twoptconv}, we prove Theorem \ref{thm:convergence}.

\section{The decomposition theorem of $p-$adic M\"obius maps}

Two points $\alpha,\beta$ are called {\it antipodal} points if there exists an element $u\in \mathrm{PSL}(2,\mathcal{O}_{p})$ such that $u(0)=\alpha,u(\infty)=\beta$.

\begin{thm}\label{thm:de} For any $p$-adic M\"obius map $g$, there exists an element $u\in \mathrm{PSL}(2,\mathcal{O}_{p})$ such that $g=uf$, where either $f$ is a loxodromic element with antipodal fixed points, or $f=I$.

\end{thm}

\begin{proof} If $g$ is a loxodromic element or elliptic element,  by Lemma \ref{thm:1.1}, there exist two involutions $a,b$ such that $g=ab$,  the (tailed) axes of $a$ and $b$ are orthogonal to the axis of $g$, and the (tailed) axis of $a$ containing $\zeta_{Gauss}$. Let $\alpha$, $\beta$ be the fixed points of $b$. We claim that there exists an element in  $h\in \mathrm{PSL}(2,\mathcal{O}_{p})$ such that $h(\alpha)+h(\beta)=0$.

{\it Claim}:
Without loss of generality, we can assume that $|\alpha|\le 1$, otherwise we can consider $1/\alpha$, since $1/z\in \mathrm{PSL}(2,\mathcal{O}_{p})$. Let $u(z)=z-\alpha\in \mathrm{PSL}(2,\mathcal{O}_{p})$ which yields that $u(\alpha)=0$. Hence we can assume that $\alpha=0$. If $|\beta|\le 1$, then let $h(z)=z-\frac{\beta}{2}$ which implies that $h(0)+h(\beta)=0$.  If $|\beta|>1$, then let $h(z)=\frac{az+b}{cz+d}$ with $h(0)+h(\beta)=0$. Hence $$\frac{b}{d}+\frac{a\beta+b}{c\beta+d}=0$$ which yields that $\beta(ad+bc)+2bd=0$. Let $X_{1}=ad,X_{2}=bc,X_{3}=bd$. Hence we have three equations:

\begin{equation}
\beta(X_{1}+X_{2})+2X_{3}=0
\end{equation}

\begin{equation}
X_{1}-X_{2}=1
\end{equation}

\begin{equation}
X_{1}+X_{2}=\lambda
\end{equation}
Thus we have $X_{1}=\frac{\lambda+1}{2},X_{2}=\frac{\lambda-1}{2}$, and $X_{3}=-\frac{\lambda\beta}{2}$. Since $\frac{a}{b}=\frac{X_{1}}{X_{3}}$ and $\frac{c}{d}=\frac{X_{2}}{X_{3}}$, let $a=\frac{\lambda+1}{2}t$, $b=-\frac{\lambda\beta}{2}t$, $c=\frac{\lambda-1}{2}s$, and $d=-\frac{\lambda\beta}{2}s$ which yields that $-\frac{\lambda\beta(\lambda+1)}{4}ts+\frac{\lambda\beta(\lambda-1)}{4}ts=1$, namely $-\frac{\lambda\beta}{2}ts=1$. Obviously, $\lambda=1/\beta$, $t=2$, and $s=1$ is one of the solution of the equations with $\max\{|a|,|b|,|c|,|d|\}\le 1$, namely $h\in \mathrm{PSL}(2,\mathcal{O}_{p})$. Hence we prove the claim.

Since $h(\alpha)+h(\beta)=0$, the geodesic line connecting $0,\infty$  is orthogonal to the geodesic line connecting $h(\alpha),h(\beta)$ , and this line contains $\zeta_{Gauss}$. Let $l_{1}$ be the geodesic line which connects the endpoints of involution $a$. Since all the elements in $\mathrm{PSL}(2,\mathcal{O}_{p})$ fix the point $\zeta_{Gauss}$, the point $\zeta_{Gauss}\in h(l_{1})$. By Lemma \ref{thm:1.1}, we know that there exists an involution $c$ which fixes the point $\zeta_{Gauss}$ such that $chbh^{-1}$ is a loxodromic element whose fixed points are $0,\infty$, and $hah^{-1}c$ is an elliptic element in $\mathrm{PSL}(2,\mathcal{O}_{p})$, since $hah^{-1}c$ fixes the point $\zeta_{Gauss}$. Thus $hgh^{-1}=hah^{-1}cchbh^{-1}$ which yields that $g=h^{-1}hah^{-1}chh^{-1}(chbh^{-1})h$, where $h^{-1}hah^{-1}ch\in \mathrm{PSL}(2,\mathcal{O}_{p})$ and $h^{-1}(chbh^{-1})h$ is a loxodromic element with antipodal fixed points.

If $g$ is a parabolic element, then we can assume the fixed point of $g$ is $\infty$ after conjugating by an element in $\mathrm{PSL}(2,\mathcal{O}_{p})$ by the Claim. Thus $g(z)=ab$, where $a(z)=-z,b(z)=-z+b$. Since $a(z)=-z$ contains the point $\zeta_{Gauss}$, by similar discussion above, we can find a involution $c$ which fixes the point $\zeta_{Gauss}$ such that $az\in \mathrm{PSL}(2,\mathcal{O}_{p})$, and $cb$ is a loxodromic element with antipodal fixed points.

%Without lose of generality, we can assume that $|\alpha\beta|\ge 1$, otherwise we can consider $1/\alpha,1/\beta$, since
 %  Let $X_{1}=ac,X_{2}=ad,X_{3}=bc,X_{4}=bd$. Hence we have three equations:

%\begin{equation}
%2\alpha\beta X_{1}+(\alpha+\beta)X_{2}+(\alpha+\beta)X_{3}+2X_{4}=0
%\end{equation}

%\begin{equation}
%X_{2}-X_{3}=1
%\end{equation}

%\begin{equation}
%X_{2}X_{3}=X_{1}X_{4}
%\end{equation}

%Let $X_{2}-X_{3}=\lambda$ which yields that $X_{2}=\frac{\lambda+1}{2}$, $X_{3}=\frac{\lambda-1}{2}$. Hence $X_{1}X_{4}=\frac{\lambda^{2}-1}{4}$, and $\alpha\beta X_{1}+X_{4}=-\frac{(\alpha+\beta)\lambda}{2}$. We get

%\begin{equation}\label{equ:1}
%\alpha\beta X_{1}^{2}+\frac{\alpha+\beta}{2}X_{1}+\frac{1}{4}(\lambda^{2}-1)=0
%\end{equation}
%Let $Z_{1},Z_{2}$ be the roots of the equation \ref{equ:1}. Thus $Z_{1}Z_{2}=\frac{\frac{1}{4}(\lambda^{2}-1)}{\alpha\beta}$

\end{proof}

%For any $z$ with $|z|=1$, we have $|cz^{2}+(d-a)z-b|\le \max\{|cz^{2}|,|(d-a)z|,|-b|\}=\max\{|c|,|d-a|,|b|\}$.

Let $\mathcal{U}=\mathrm{PSL}(2,\mathcal{O}_{p})$. We define $d(g,\mathcal{U})=\inf\{\rho_{0}(g,u)|u\in \mathcal{U}\}$.

\begin{thm}\label{thm:hype}For any $p$-adic M\"obius map, either $d(g,\mathcal{U})=0$, if $g\in \mathcal{U}$, or $d(g,\mathcal{U})=1$, if $g\notin \mathcal{U}$.

\end{thm}

\begin{proof} If $g\in \mathcal{U}$, then  $d(g,\mathcal{U})=0$. If $g\notin \mathcal{U}$, then by Theorem \ref{thm:de}, there exist $u\in \mathrm{PSL}(2,\mathcal{O}_{p})$ and a loxodromic element $f$ with antipodal fixed points such that $g=uf$. Hence there exists an element $h\in \mathrm{PSL}(2,\mathcal{O}_{p})$ such that $v=hfh^{-1}=\lambda z$. Thus $d(g,\mathcal{U})=d(uf,\mathcal{U})=d(v,\mathcal{U})$.

For any $s(z)=\frac{az+b}{cz+d}\in \mathrm{PSL}(2,\mathcal{O}_{p})$, since $ad-bc=1$ and $\max\{|a|,|b|,|c|,|d|\}\le 1$, we have that $|\frac{b}{d}|=1$, or $|\frac{a}{c}|=1$, or $|\frac{b}{d}|>1, |\frac{a}{c}|<1$, otherwise, if $|\frac{b}{d}|>1$ and $ |\frac{a}{c}|>1$, then $|d|<|b|\le1$ and $|c|<|a|\le1$ which yields that $|ad-bc|\le \max\{|ad|,|bc|\}<1$. This contradicts $ad-bc=1$. Other cases are similar. If $|\frac{b}{d}|=1$, then $\rho_{v}(v(0),s(0))=1$. If $|\frac{a}{c}|=1$, then $\rho_{v}(v(\infty),s(\infty))=1$. If $|\frac{b}{d}|>1, |\frac{a}{c}|<1$, then $1\ge |b|>|d|$ and $|a|<|c|\le1$. Since $ad-bc=1$, we have $|b|=|c|=1>\max\{|a|,|d|\}$. This implies that $|\frac{d}{c}|<1$. $$\rho_{v}(v(-\frac{d}{c}),s(-\frac{d}{c}))=\frac{|\lambda c(-\frac{d}{c})^{2}+(\lambda d-a)(-\frac{d}{c})-b|}{\max\{1,|-\frac{d}{c}|\}\max\{|a(-\frac{d}{c})+b|,|c(-\frac{d}{c})+d|\}}=\frac{|1/c|}{|1/c|}=1.$$
This yields that for any $p$-adic M\"obius map $g\notin \mathcal{U}$, and any $s\in \mathcal{U}$, $\rho_{0}(g,s)=1$, namely $d(g,\mathcal{U})=1$.

\end{proof}

Let $G$ be a subgroup of $\mathrm{PSL}(2,\mathbb{C}_{p})$. We say that $G$ a discrete subgroup if there exists a positive number $\varepsilon$ such that for any non-unit element $g$ with $\parallel g-I\parallel>\varepsilon$.

\begin{thm}\label{thm:dis}If $G$ is a subgroup of $\mathrm{PSL}(2,\mathbb{C}_{p})$ and $G\cap \mathcal{U}=I$, then $G$ is a discrete subgroup.

\end{thm}

\begin{proof} By Theorem \ref{thm:matr} and Theorem \ref{thm:hype}, we have $d(f,\mathcal{U})\le \rho_{0}(f,I)\le \parallel f-I\parallel$. Since $G\cap \mathcal{U}=I$, for  any nonunit element $f$,  $\parallel f-I\parallel\ge 1$.

\end{proof}

\begin{cor} If a subgroup $G\subset \mathrm{PSL}(2,\mathbb{C}_{p})$ contains unit element or loxodromic element only, then $G$ is a discrete subgroup.

\end{cor}

\begin{proof} Since each loxodromic element does not belong to $\mathcal{U}$, by Theorem \ref{thm:dis}, we can get the conclusion directly.

\end{proof}

%%%%%%%%%%%%%%%%%%%%%%%%%%%%%%%%%%%%%%%%%%%%%%%%%%%%%%%%%%%%%%%%%%%%%%%%%%%%%%%%%%%%%%%%%%%%%%%%%%%%%%%%%%%%%%%%%%%%%

\end{document}